\documentclass[11pt]{article}

\usepackage[margin=1.1in]{geometry}
\usepackage[T1]{fontenc}
\usepackage[utf8]{inputenc}
\usepackage{lmodern}
\IfFileExists{microtype.sty}{\usepackage{microtype}}{}
\usepackage{amsmath,amssymb,amsthm}
\usepackage{placeins}
\usepackage{xcolor}
\usepackage{graphicx}
\usepackage[numbers,sort&compress]{natbib}
\usepackage[
  colorlinks=true,
  linkcolor=blue,
  citecolor=blue,
  urlcolor=blue,
  pdftitle={Mirror and knockoff+ thresholds under dependence},
  pdfauthor={Xianyang Zhang}
]{hyperref}

\newcommand{\E}{\mathbb{E}}

\newcommand{\FDP}{\operatorname{FDP}}
\newcommand{\FDR}{\operatorname{FDR}}

\newcommand{\ind}{\mathbf{1}}
\newcommand{\Hnull}{\mathcal H_0}
\newcommand{\Halt}{\mathcal H_1}
\newcommand{\Rset}{\mathcal R}
\newcommand{\Prb}{\mathbb P}
\newcommand{\barPhi}{\overline\Phi}
\newcommand{\Power}{\operatorname{Power}}
\allowdisplaybreaks

\newcommand{\rev}[1]{{\color{black}#1}}
\newenvironment{revision}{\begingroup\color{black}}{\endgroup}
\newcommand{\newrev}[1]{{\color{black}#1}}
\newenvironment{newrevision}{\begingroup\color{black}}{\endgroup}

\ExplSyntaxOn
\seq_new:N \g_mirror_appendix_a_seq
\seq_new:N \g_mirror_appendix_b_seq
\seq_new:N \g_mirror_appendix_c_seq
\cs_new_protected:Npn \mirror_store_proof:nnnn #1#2#3#4
 {
  \str_case:nnF {#1}
   {
    {A}{\seq_gput_right:Nn \g_mirror_appendix_a_seq
       {\begin{revision}\subsection{#2}\label{#3}\begin{proof}#4\end{proof}\end{revision}}}
    {B}{\seq_gput_right:Nn \g_mirror_appendix_b_seq
       {\begin{revision}\subsection{#2}\label{#3}\begin{proof}#4\end{proof}\end{revision}}}
    {C}{\seq_gput_right:Nn \g_mirror_appendix_c_seq
       {\begin{revision}\subsection{#2}\label{#3}\begin{proof}#4\end{proof}\end{revision}}}
   }
   {\PackageError{mirror-paper}{Unknown deferred-proof appendix #1}{Use A, B, or C.}}
 }
\cs_new_protected:Npn \mirror_store_legacy_proof:nnnn #1#2#3#4
 {
  \str_case:nnF {#1}
   {
    {A}{\seq_gput_right:Nn \g_mirror_appendix_a_seq
       {\begin{revision}\subsection{#2}\label{#3}\end{revision}\begin{proof}#4\end{proof}}}
    {B}{\seq_gput_right:Nn \g_mirror_appendix_b_seq
       {\begin{revision}\subsection{#2}\label{#3}\end{revision}\begin{proof}#4\end{proof}}}
    {C}{\seq_gput_right:Nn \g_mirror_appendix_c_seq
       {\begin{revision}\subsection{#2}\label{#3}\end{revision}\begin{proof}#4\end{proof}}}
   }
   {\PackageError{mirror-paper}{Unknown deferred-proof appendix #1}{Use A, B, or C.}}
 }
\NewDocumentEnvironment{deferredproof}{m m m +b}
 {
  \par\noindent\rev{\emph{Proof.}~See~Appendix~\ref{#3}.}\par
  \mirror_store_proof:nnnn {#1}{#2}{#3}{#4}
 }
 {}
\NewDocumentEnvironment{storedproof}{m m m +b}
 {
  \mirror_store_proof:nnnn {#1}{#2}{#3}{#4}
 }
 {}
\NewDocumentEnvironment{deferredlegacyproof}{m m m +b}
 {
  \par\noindent\rev{\emph{Proof.}~See~Appendix~\ref{#3}.}\par
  \mirror_store_legacy_proof:nnnn {#1}{#2}{#3}{#4}
 }
 {}
\cs_new_protected:Npn \mirror_store_new_proof:nnnn #1#2#3#4
 {
  \str_case:nnF {#1}
   {
    {A}{\seq_gput_right:Nn \g_mirror_appendix_a_seq
       {\begin{newrevision}\subsection{#2}\label{#3}\begin{proof}#4\end{proof}\end{newrevision}}}
    {B}{\seq_gput_right:Nn \g_mirror_appendix_b_seq
       {\begin{newrevision}\subsection{#2}\label{#3}\begin{proof}#4\end{proof}\end{newrevision}}}
    {C}{\seq_gput_right:Nn \g_mirror_appendix_c_seq
       {\begin{newrevision}\subsection{#2}\label{#3}\begin{proof}#4\end{proof}\end{newrevision}}}
   }
   {\PackageError{mirror-paper}{Unknown deferred-proof appendix #1}{Use A, B, or C.}}
 }
\NewDocumentEnvironment{deferrednewproof}{m m m +b}
 {
  \par\noindent\newrev{\emph{#2.}~See~Appendix~\ref{#3}.}\par
  \mirror_store_new_proof:nnnn {#1}{#2}{#3}{#4}
 }
 {}
\NewDocumentCommand{\printappendixproofs}{m}
 {
  \str_case:nnF {#1}
   {
    {A}{\seq_use:Nn \g_mirror_appendix_a_seq {}}
    {B}{\seq_use:Nn \g_mirror_appendix_b_seq {}}
    {C}{\seq_use:Nn \g_mirror_appendix_c_seq {}}
   }
   {\PackageError{mirror-paper}{Unknown deferred-proof appendix #1}{Use A, B, or C.}}
 }
\ExplSyntaxOff

\theoremstyle{definition}

\newtheorem{assumption}{Assumption}

\theoremstyle{plain}
\newtheorem{theorem}{Theorem}
\newtheorem{proposition}{Proposition}
\newtheorem{lemma}{Lemma}
\newtheorem{corollary}{Corollary}

\theoremstyle{remark}
\newtheorem{remark}{Remark}

\title{Mirror and knockoff+ thresholds under dependence}
\author{Xianyang Zhang\thanks{Email: zhangxiany@stat.tamu.edu}\\[2pt]
\normalsize Department of Statistics, Texas A\&M University}
\date{August 2026}

\begin{document}

\maketitle

\begin{abstract}
\begin{revision}
Many multiple-testing procedures control the false discovery rate (FDR) by comparing the two tails of a null distribution. At a fixed cutoff, marginal symmetry makes this natural. Mirror and knockoff+ thresholds select the cutoff from the same data, so the standard finite-sample guarantee uses a stronger property: conditional on magnitudes and nonnull scores, null signs are independent fair coins.

Failure can be severe without this property. Models satisfying positive regression dependence on a subset (PRDS) can have exactly uniform null $p$-values and large FDR. Under every fixed positive Gaussian equicorrelation, the all-null FDR converges to one half. Opposing loadings in Gaussian factor models can make FDR and power arbitrarily close to one; near-total failure also occurs for exchangeable, pairwise-uncorrelated scores. At a nominal input level $q<1/2$, no deterministic rule based only on the two tail counts can both reject and control FDR uniformly over our class if more discoveries or fewer controls cannot make rejection harder.

We give finite-sample repairs based on joint sign information. Conditional sign odds may be bounded outside an exceptional event or averaged over negative controls; neither route uniformly dominates, and the integrated bounds are sharp. Independent calibration data or a specified Gaussian joint model yield valid adjusted levels. Simulations show that integration retains more power under diffuse Gaussian dependence, whereas exceptional-event calibration is more powerful when very large odds occur only for rare aligned signs; unadjusted FDR exceeds the target in both settings. Covariance alone is insufficient outside a specified joint model. Thus the relevant boundary is not marginal symmetry but joint information that remains valid after adaptive cutoff selection.
\end{revision}
\end{abstract}

\begin{center}
\textit{Keywords}: False discovery rate, Dependence, Knockoff+, Mirror procedures, PRDS, Sign flips.
\end{center}

\section{Introduction}

When many hypotheses are tested at once, we would like to report as many discoveries as possible without reporting too many false ones.  The false discovery rate (FDR) measures the expected proportion of false discoveries among all reported discoveries \citep{benjamini-hochberg-1995}.

A common way to estimate the number of false discoveries is to reserve part of the null distribution as a control region.  For $p$-values, observations near one can be used to estimate how many null $p$-values lie near zero.  For signed scores, large negative values can be used as controls for large positive values.  This idea appears in Storey-type estimators \citep{storey-2002,storey-taylor-siegmund-2004}, the knockoff+ filter \citep{barber-candes-2015,candes-et-al-2018}, and several procedures based on mirror statistics \citep{xing-zhao-liu-2023,dai-et-al-2023a,dai-et-al-2023b,deng-he-zhang-2024}.

At a fixed threshold, the idea seems straightforward.  Under a null hypothesis, the two sides of a symmetric distribution should be equally likely.  The number of observations on the negative, or control, side should therefore tell us roughly how many null observations appear on the positive, or discovery, side.

There is, however, an important complication: the threshold is chosen after looking at the data.  Suppose many null scores share a common source of noise.  On one realization, this common noise may push a large number of them to the positive side at the same time.  The procedure then sees many apparent discoveries and few negative controls.  This can happen even when every individual null score is symmetric about zero.

Valid knockoff statistics avoid this problem through a much stronger symmetry.  Conditional on the score magnitudes and on \rev{the scores for the nonnull hypotheses (the genuine signals)}, the signs of the null statistics behave as independent fair coin flips \citep{barber-candes-2015,candes-et-al-2018}.  Thus, after the hypotheses are ordered by their score magnitudes, each newly revealed null sign is still a fresh coin flip.  \rev{This keeps the FDR argument valid when the process stops at a data-chosen cutoff, which is the role of optional stopping in the knockoff+ proof.}

\rev{Marginal symmetry alone does not imply this property.  Nor do symmetry of the whole score vector about zero, Gaussianity, exchangeability (invariance under permuting the coordinates), or positive regression dependence on a subset (PRDS) on their own.  PRDS is a standard positive-dependence condition; roughly, it says that small $p$-values tend to occur together.}  It is sufficient for the Benjamini--Hochberg procedure \citep{benjamini-yekutieli-2001}, but it need not be sufficient for a procedure that adaptively compares the two tails of the null distribution.

This paper asks a simple question:
\begin{quote}
What can happen when the mirror or knockoff+ threshold is applied to dependent scores or $p$-values that do not have the conditional sign-flip property?
\end{quote}
\begin{revision}
The answer has two parts.  First, dependence can push many null scores to the discovery side together while every null distribution remains individually well behaved.  Second, the two tail counts alone cannot tell whether the resulting imbalance is an ordinary fluctuation or a consequence of shared dependence.  A valid repair must therefore account for dependence among the null signs, not just use a more conservative function of the same two counts.
\end{revision}

\begin{revision}
We distinguish the threshold from the procedure that justifies it.  We call the upper-versus-lower-tail rule for $p$-values the \emph{mirror threshold}, and its signed-score form the \emph{knockoff+ threshold}.  The \emph{knockoff+ filter} is the full procedure that constructs knockoff variables and scores with the required conditional sign-flip property.  Our examples concern the threshold applied to generic dependent inputs, not the validity of the full knockoff filter.
\end{revision}

\subsection{A peek at the results}

\begin{revision}
The first results show two ways dependence can defeat the threshold.  In the PRDS construction, a hidden binary state moves null $p$-values together while leaving each one uniform.  The failure persists with genuine alternatives, and binomial calculations give FDR and power bounds.  A second, sparse mechanism occurs when an early run of positive signs makes the ordered threshold pass even though no fixed cutoff crosses asymptotically.

For Gaussian scores, write
\[
 W_i=\sqrt\rho\,Z_0+\sqrt{1-\rho}\,Z_i,
\]
where $Z_0,Z_1,\ldots$ are independent standard normal variables and $Z_0$ is the shared factor.  The coefficient on $Z_0$ is a factor loading.  With equal positive loadings and all hypotheses null, every fixed $\rho>0$ gives FDR converging exactly to one half: conditional on $Z_0=z\ne0$, the rejection probability converges to $\ind\{z>0\}$.  Under the stated conditional-independence assumption on nonnull scores, same-sign null loadings also give the finite-sample upper bound $(1+q)/2$.  Opposing null loadings can instead make FDR and power approach one, even with standardized nulls and small pairwise correlations.  An exchangeable construction gives near-total failure with zero pairwise correlation.

No useful correction can rely only on the discovery and control counts over the broad class studied here.  Valid adjustments instead use joint null-sign information: total variation gives an additive FDR bound, while coordinatewise conditional odds yield complementary exceptional-event and integrated negative-side bounds.  An explicit construction attains the active-set integrated bound and makes its level-independent upper bound asymptotically sharp.  A specified Gaussian covariance model makes both odds-based factors computable.  Two dependence sweeps confirm that neither adjustment uniformly dominates: integration is more powerful under a diffuse Gaussian shared factor, while exceptional-event calibration is more powerful under a rare aligned-sign mixture.  In both sweeps, the unadjusted FDR exceeds $0.1$.
\end{revision}

\subsection{Relationship with existing methods}

\begin{revision}
At a candidate cutoff, the mirror rule estimates false discoveries by adding one to the control-side count.  This is closely related to a modified Storey estimate, which also uses large $p$-values to estimate false discoveries.  Fixed-cutoff versions can fail under strong dependence \citep{blanchard-roquain-2009}; here the cutoff moves with the data.

Existing mirror methods obtain validity in several ways.  Symmetrized data aggregation uses sample splitting, screening, and aggregation \citep{du-guo-sun-zou-2023}.  Other procedures combine splitting or added Gaussian noise with weak-dependence assumptions \citep{xing-zhao-liu-2023,dai-et-al-2023a,dai-et-al-2023b}; standard weak-dependence conditions, known as mixing conditions, justify related covariate-adaptive rules \citep{zhang-chen-2022}.

Masking part of each statistic and revealing it gradually provide another route, as in AdaPT, STAR, and the joint mirror procedure \citep{lei-fithian-2018,lei-ramdas-fithian-2021,deng-he-zhang-2024}.  Conditional calibration instead uses a sufficiently specified joint null law \citep{fithian-lei-2022}; conditioning can also relax the assumptions needed for knockoff construction \citep{huang-janson-2020}.  Recent work develops further mirror procedures under explicit dependence conditions \citep{guo-lin-liu-2026,wu-yuan-jiang-2026}.

Two comparisons are especially close to our results.  Robust knockoff theory measures how symmetry errors affect the FDR one coordinate at a time, conditioning on the others \citep{barber-candes-samworth-2020}; our conditional-odds bound uses the same principle.  \citet{lei-2026} shows that, for a broad class of one-common-factor Gaussian models under the Benjamini--Hochberg procedure, the worst-case FDR scales as $q\sqrt{\log(1/q)}$, up to constant factors, as $q\downarrow0$.  In contrast, Theorem~\ref{thm:gaussian-nearone} implies that the worst-case FDR of the mirror threshold need not tend to zero.  In particular, no universal constant $C$ can satisfy $\FDR\le Cq$ uniformly over $q$, dimension, mean vectors, and Gaussian covariance matrices.

\end{revision}
\section{Mirror and knockoff+ thresholds}
\label{sec:threshold}

Let $H_1,\ldots,H_m$ be hypotheses, let $\Hnull$ be the set of true nulls, \rev{let $\Halt=\{1,\ldots,m\}\setminus\Hnull$ be the set of alternatives,} and let $\widehat\Rset\subseteq\{1,\ldots,m\}$ denote a data-dependent rejection set.  \rev{Write $a\vee b=\max\{a,b\}$, and let $\ind\{A\}$ denote the indicator of an event $A$.}  Following \citet{benjamini-hochberg-1995}, the false discovery proportion (FDP) and FDR are
\[
 \FDP=\frac{|\widehat\Rset\cap\Hnull|}{|\widehat\Rset|\vee1},
 \qquad
 \FDR=\E(\FDP).
\]
When $n_1=|\Halt|>0$, we also report the usual average power
\[
 \Power
 =\E\left\{\frac{|\widehat\Rset\cap\Halt|}{n_1}\right\}.
\]
Fix a nominal input level $q\in(0,1)$.  For $p$-values $P_1,\ldots,P_m$ and $0<u\le1/2$, define
\[
 R(u)=\#\{i:P_i\le u\},
 \qquad
 L(u)=\#\{i:P_i\ge1-u\}.
\]
Let
\[
 \mathcal U=\{\min(P_i,1-P_i):1\le i\le m\}\cap(0,1/2]
\]
and set
\begin{equation}
 \widehat u
 =\max\left\{u\in\mathcal U:
       \frac{1+L(u)}{R(u)\vee1}\le q\right\},
 \label{eq:pvalue-threshold}
\end{equation}
with $\widehat u=0$ if the set is empty.  The mirror procedure rejects $P_i\le\widehat u$.

For signed scores $W_1,\ldots,W_m$, a large positive value favors rejection and a large negative value serves as a control.  For $t>0$, write
\[
 N_+(t)=\#\{i:W_i\ge t\},
 \qquad
 N_-(t)=\#\{i:W_i\le-t\}.
\]
The signed-score rule below is the knockoff+ threshold of \citet{barber-candes-2015}; after ordering the hypotheses by $|W_i|$, it is also the Selective SeqStep+ rule in that paper.  \rev{Let $\mathcal T=\{|W_i|:|W_i|>0\}$ be the set of candidate thresholds.}  The threshold is
\begin{equation}
 \widehat t
 =\min\left\{t\in\mathcal T:
       \frac{1+N_-(t)}{N_+(t)\vee1}\le q\right\}.
 \label{eq:threshold}
\end{equation}
If the set is empty, let $\widehat t=\infty$.  The selected set is
\[
 \widehat\Rset=\{i:W_i\ge\widehat t\}.
\]
We use $\le q$ in both definitions; replacing it by $<q$ changes only boundary cases.

\begin{proposition}[Equivalence of the two forms]
\label{prop:equivalence}
Let $W_i=1/2-P_i$.  Apart from deterministic tie conventions, the mirror threshold \eqref{eq:pvalue-threshold} and the knockoff+ threshold \eqref{eq:threshold} have the same rejection set.
\end{proposition}

\begin{proof}
For $t=1/2-u$,
\[
 P_i\le u\quad\Longleftrightarrow\quad W_i\ge t,
 \qquad
 P_i\ge1-u\quad\Longleftrightarrow\quad W_i\le-t.
\]
\rev{Moreover, if $u=\min(P_i,1-P_i)$, then
$t=1/2-u=|1/2-P_i|=|W_i|$.}  Thus the largest admissible $u$ is the smallest admissible $t$.
\end{proof}

More generally, if a null score has a continuous distribution function $F_0$ symmetric about zero, then $P_i=1-F_0(W_i)$ gives the same correspondence after a monotone change of threshold.  This statement is directional: a two-sided $p$-value discards the sign and is not automatically equivalent to a signed-score rule.

Under the global null, every rejection is false, and therefore
\begin{equation}
 \FDR=\Prb(\widehat\Rset\ne\varnothing).
 \label{eq:global-null}
\end{equation}
\begin{revision}
This identity simplifies the all-null calculations, but the main conclusions do not rely on it.  Sections~\ref{sec:nongaussian}, \ref{sec:fixedrho}, and \ref{sec:nearone} give mixed configurations with genuine alternatives, quantitative power statements, or finite-sample bounds allowing arbitrary nonnull scores.
\end{revision}

\section{Why valid knockoff statistics work}
\label{sec:valid}

The key symmetry in knockoff theory is coordinatewise and conditional.  It is much stronger than saying that each null score is marginally symmetric.

\begin{assumption}[Conditional null sign flips]
\label{ass:signflip}
Conditional on
\[
 \mathcal G
 =\sigma\bigl(|W_1|,\ldots,|W_m|,(W_i)_{i\notin\Hnull}\bigr),
\]
the signs $\{\operatorname{sign}(W_i):i\in\Hnull,\ |W_i|>0\}$ of the nonzero null statistics are mutually independent, each taking the values $+1$ and $-1$ with probability $1/2$.  A null coordinate with $W_i=0$ is never selected by \eqref{eq:threshold} and may be ignored.
\end{assumption}

The magnitudes may be arbitrarily dependent, and they may depend on all the nonnull scores.  Only the signs of the nonzero null statistics must remain conditionally independent and fair.

\begin{proposition}[Known knockoff+ guarantee]
\label{prop:known}
Under Assumption~\ref{ass:signflip}, the rule \eqref{eq:threshold} satisfies $\FDR\le q$.
\end{proposition}

This is the standard knockoff+ argument; see \citet{barber-candes-2015} for fixed-X knockoffs and \citet{candes-et-al-2018} for model-X knockoffs.  We therefore omit the proof.  It may help to picture what the proof does.  Order the null scores from largest magnitude to smallest and reveal their signs one at a time.  Under Assumption~\ref{ass:signflip}, every revealed null sign is a fresh fair coin.  \rev{The plus one makes the relevant ratio a reverse-time supermartingale: its conditional mean cannot increase as the ordered signs are revealed backward.  Optional stopping is therefore valid.}

\begin{revision}
A related leave-one-out argument yields the bound below without requiring mutual independence of the null signs.  In the lemma, $B_j$ represents a ``good'' event on which the positive sign of null $j$ has conditional odds at most $\Gamma$ times the negative sign.

\begin{lemma}[Leave-one-out sign-odds bound]
\label{lem:leave-one-out-odds}
For every null $j$, let $W_{-j}$ denote the vector of all scores except $W_j$, and let $B_j\in\sigma(|W_j|,W_{-j})$.  Suppose that, for some $\Gamma\ge1$,
\begin{equation}
 \Prb(W_j>0,B_j\mid |W_j|,W_{-j})
 \le
 \Gamma\Prb(W_j<0\mid |W_j|,W_{-j})
 \quad\text{almost surely}.
 \label{eq:leave-one-out-odds}
\end{equation}
Then the contribution to the FDR from selected nulls satisfying $B_j$ is at most $\Gamma q$:
\[
 \E\left\{
  \frac{\sum_{j\in\Hnull}\ind\{W_j\ge\widehat t,B_j\}}
       {N_+(\widehat t)\vee1}
 \right\}
 \le \Gamma q.
\]
\end{lemma}

\begin{proof}
This is the leave-one-out argument in the proof of Theorem~2 of \citet{barber-candes-samworth-2020}, specialized to the knockoff+ threshold.  Although that theorem is stated for model-X knockoffs, the part of their proof following condition~(16) in that paper uses only the signed score vector, the conditional inequality in that condition, and the knockoff+ threshold.  Define $E_j=\log\Gamma$ on $B_j$ and $E_j=+\infty$ on $B_j^c$.  For $\varepsilon<\log\Gamma$, the event $\{E_j\le\varepsilon\}$ is empty.  For $\varepsilon\ge\log\Gamma$, condition \eqref{eq:leave-one-out-odds} gives the required conditional inequality because $\Gamma\le e^\varepsilon$.  The same leave-one-out proof therefore bounds the stated FDR contribution by $\Gamma q$.
\end{proof}

The lemma immediately covers null signs that are conditionally biased toward the control side.

\begin{proposition}[Conservative conditional null signs]
\label{prop:conservative-signs}
Suppose that, conditional on
\[
 \mathcal G
 =\sigma\bigl(|W_1|,\ldots,|W_m|,(W_i)_{i\notin\Hnull}\bigr),
\]
the signs of the nonzero null statistics are mutually independent and satisfy
\[
 \Prb(W_i>0\mid\mathcal G)\le\frac12,
 \qquad i\in\Hnull.
\]
Then the knockoff+ threshold \eqref{eq:threshold} satisfies $\FDR\le q$.
\end{proposition}

\begin{proof}
Fix a null $j$.  Conditioning on $(|W_j|,W_{-j})$ reveals $\mathcal G$ and the signs of the other null scores.  Conditional independence given $\mathcal G$ therefore gives
\[
 \Prb(W_j>0\mid |W_j|,W_{-j})
 =\Prb(W_j>0\mid\mathcal G)
 \le\frac12.
\]
Hence \eqref{eq:leave-one-out-odds} holds with $B_j$ equal to the whole sample space and $\Gamma=1$.  Lemma~\ref{lem:leave-one-out-odds} gives $\FDR\le q$.
\end{proof}

The distinction is visible with only two Gaussian scores.  Suppose they are standard normal with correlation $\rho>0$.  For $u,v>0$,
\[
 \Prb\{\operatorname{sign}(W_1)=\operatorname{sign}(W_2)
       \mid |W_1|=u,|W_2|=v\}
 =\frac{1}{1+\exp\{-2\rho uv/(1-\rho^2)\}}>\frac12.
\]
Once the magnitudes are known, equal signs are more likely than opposite signs.  Global sign symmetry therefore does not give coordinatewise sign flips.
\end{revision}

Valid knockoff statistics obtain Assumption~\ref{ass:signflip} from two ingredients.  Swapping a null variable with its knockoff leaves the relevant joint law unchanged, and an antisymmetric feature statistic changes sign under that swap \citep{barber-candes-2015,candes-et-al-2018}.  An arbitrary symmetric or Gaussian score vector inserted into \eqref{eq:threshold} has no such guarantee.

\section{A simple failure mechanism}
\label{sec:trigger}

The following observation drives \rev{several of the counterexamples}.  If a sufficiently large block of null scores moves to the positive side together, the threshold sees many discoveries and few controls.

\begin{lemma}[A positive block forces a rejection]
\label{lem:trigger}
Suppose all $m$ hypotheses are null.  Let $B\subseteq\{1,\ldots,m\}$ contain $n$ indices, and write $s=m-n$.  If
\begin{equation}
 \frac{1+s}{n}<q,
 \label{eq:blockcondition}
\end{equation}
then
\[
 \FDR\ge \Prb(W_i>0\text{ for every }i\in B).
\]
\end{lemma}

\begin{proof}
On the displayed event, let $t_B=\min_{i\in B}W_i>0$.  At least $n$ scores are above $t_B$, while at most the $s$ scores outside $B$ can be below $-t_B$.  Hence
\[
 \frac{1+N_-(t_B)}{N_+(t_B)\vee1}
 \le \frac{1+s}{n}<q.
\]
The two counts are constant between consecutive observed magnitudes, so some threshold in $\mathcal T$ also passes.  The procedure makes a rejection.  Since every hypothesis is null, the false discovery proportion is one.  Equation~\eqref{eq:global-null} proves the claim.
\end{proof}

\rev{When $B$ contains all $m$ scores, condition \eqref{eq:blockcondition} is simply $m>1/q$.  Any all-null model that assigns probability greater than $q$ to the event that every score is positive must therefore violate FDR control.}

\section{A non-Gaussian PRDS example}
\label{sec:nongaussian}

\begin{revision}
PRDS is a standard positive-dependence condition under which the Benjamini--Hochberg procedure remains valid \citep{benjamini-hochberg-1995,benjamini-yekutieli-2001}.  Informally, conditioning on a larger null $p$-value should not make an increasing event less likely.  Formally, a vector $P$ is PRDS if, for every coordinatewise increasing measurable set $D$ and every null index $i$, the map
\end{revision}
\[
 u\longmapsto \Prb(P\in D\mid P_i=u)
\]
is nondecreasing.

\begin{revision}
PRDS protects Benjamini--Hochberg, but it does not automatically protect other adaptive rules; correlated two-sided Gaussian tests provide one such warning \citep{dobriban-2026}.  The mirror threshold fails even in the following simple PRDS model.  A binary latent variable makes all coordinates favor the same half of the unit interval.  Conditional on that variable, the coordinates are independent and continuous.
\end{revision}

\begin{proposition}[Uniform PRDS $p$-values with full support]
\label{prop:nongaussian}
Fix $q\in(0,1/2)$.  There are an integer $m$ and an all-null vector $P=(P_1,\ldots,P_m)$ such that each $P_i$ is uniform on $(0,1)$, the joint density is strictly positive on $(0,1)^m$, and $P$ is PRDS, but the mirror threshold \eqref{eq:pvalue-threshold} has FDR greater than $q$.  The coordinates are almost surely distinct.
\end{proposition}

\begin{proof}
Choose $m>1/q$ and let $H$ be a fair Bernoulli variable.  Given $H=h$, let $P_1,\ldots,P_m$ be independent with density $f_h$.  For $0<\varepsilon<1/2$, set
\[
 f_0(u)=
 \begin{cases}
  2(1-\varepsilon),&0<u<1/2,\\
  2\varepsilon,&1/2<u<1,
 \end{cases}
 \qquad
 f_1(u)=f_0(1-u).
\]
Because $(f_0+f_1)/2=1$, every $P_i$ is uniform.  The joint density
\[
 g(p_1,\ldots,p_m)
 =\frac12\prod_{i=1}^m f_0(p_i)
  +\frac12\prod_{i=1}^m f_1(p_i)
\]
is strictly positive on the open cube.  Absolute continuity gives no ties.

We next check PRDS.  The conditional law under $H=1$ is stochastically larger than the law under $H=0$, and
\[
 w(u):=\Prb(H=1\mid P_i=u)=
 \begin{cases}
  \varepsilon,&u<1/2,\\
  1-\varepsilon,&u>1/2,
 \end{cases}
\]
is nondecreasing.  \rev{For an increasing set $D$, let $P_{-i}$ denote $P$ with coordinate $i$ removed, and put}
\[
 g_h(u)=\Prb\{(u,P_{-i})\in D\mid H=h\},
\]
\rev{Both $g_0$ and $g_1$ are nondecreasing, and stochastic ordering gives $g_1\ge g_0$.  Therefore}
\[
 Q(u):=\Prb(P\in D\mid P_i=u)
 =\{1-w(u)\}g_0(u)+w(u)g_1(u)
\]
is nondecreasing.  Indeed, for $u<v$,
\begin{align*}
 Q(v)-Q(u)
 ={}&\{1-w(v)\}\{g_0(v)-g_0(u)\}
     +w(v)\{g_1(v)-g_1(u)\}\\
 &+\{w(v)-w(u)\}\{g_1(u)-g_0(u)\}\ge0.
\end{align*}
Thus $P$ is PRDS.

Finally,
\[
 \Prb(P_1<1/2,\ldots,P_m<1/2)
 =\frac12\{(1-\varepsilon)^m+\varepsilon^m\}.
\]
On this event all scores $W_i=1/2-P_i$ are positive, so Lemma~\ref{lem:trigger} applies.  The displayed probability tends to $1/2$ as $\varepsilon\downarrow0$ and can therefore be made larger than $q$.
\end{proof}

\rev{The exact calculation is easiest after separating each score into its magnitude and sign.  Take $q=0.1$ and $m=11$, and let}
\[
 M_i=|W_i|,\qquad S_i=\ind\{W_i>0\},
\]
and let $\pi$ be the random permutation that orders the magnitudes, $M_{\pi(1)}>\cdots>M_{\pi(11)}$.  Conditional on $H=h$, the pairs $(M_i,S_i)$ are independent across $i$, the sign $S_i$ is independent of the magnitude $M_i$, and
\[
 S_i\sim\operatorname{Bernoulli}(p_h),
 \qquad p_0=1-\varepsilon,\quad p_1=\varepsilon.
\]
Because $\pi$ is determined only by the magnitudes, the reordered signs $S_{\pi(1)},\ldots,S_{\pi(11)}$ remain independent Bernoulli$(p_h)$ variables conditional on $H=h$.

At a threshold containing $k$ ordered scores, the smallest possible knockoff+ ratio is $1/k$, attained when all $k$ signs are positive.  Thus no threshold containing at most nine scores can pass.  A threshold containing ten scores passes exactly when the ten largest-magnitude scores are all positive, in which case the ratio is $1/10=0.1$.  A threshold containing all eleven scores can pass only when all eleven signs are positive, an event already contained in the preceding one.  Hence
\begin{equation}
 \FDR
 =\frac12\{(1-\varepsilon)^{10}+\varepsilon^{10}\}.
 \label{eq:prds-exact}
\end{equation}
For $\varepsilon=0.1$, this is approximately $0.1743$.  With the strict convention $<q$ in place of $\le q$, ten positive signs no longer pass at the boundary, so the exponent becomes eleven and the FDR is approximately $0.1569$.

\begin{revision}
This calculation also shows where the dependence enters.  Given $H$, the signs are independent.  After $H$ is averaged out, every sign remains marginally fair, but for $i\ne j$,
\end{revision}
\[
 \operatorname{Cov}(S_i,S_j)=\left(\frac12-\varepsilon\right)^2>0.
\]
\begin{revision}
The hidden state therefore creates coordinated signs without disturbing the individual null distributions.

The construction has no ties or lower-dimensional support: conditionally the coordinates are continuous, and unconditionally the density is positive on the whole cube.
\end{revision}
\rev{The corresponding scores have symmetric uniform margins, and $W$ and $-W$ have the same distribution.}
\begin{revision}
The jump in the density at $1/2$ can be smoothed without removing the FDR violation, although the exact formula \eqref{eq:prds-exact} uses the piecewise-constant form.

\subsection{Mixed nulls and alternatives}
\label{subsec:mixed-prds}

The counterexample above uses only null hypotheses.  We now ask whether the failure persists when genuine alternatives are present.  It does: we allow a positive fraction of alternatives and track both FDR and power.  Keep the same dependent nulls and add independent alternatives that favor the discovery half.  For $1/2<r<1$, let
\[
 g_r(u)=
 \begin{cases}
  2r,&0<u<1/2,\\
  2(1-r),&1/2<u<1,
 \end{cases}
 \qquad \frac12<r<1.
\]
A variable with density $g_r$ falls below $1/2$ with probability $r$.  It is therefore stochastically smaller than uniform, while retaining positive density on the whole unit interval.

\begin{proposition}[Exact mixed-configuration lower bounds]
\label{prop:mixed-prds}
Fix integers $n_0,n_1\ge1$, put $m=n_0+n_1$, and fix $0<\varepsilon<1/2$ and $1/2<r<1$.  Generate $n_0$ null $p$-values from the latent-state model in Proposition~\ref{prop:nongaussian}, using the same fair Bernoulli variable $H$, and generate $n_1$ independent alternative $p$-values with density $g_r$, independently of $H$ and of the nulls.  Then the null $p$-values are exactly uniform, the joint density is strictly positive on $(0,1)^m$, and the full vector is PRDS with respect to the null coordinates.

Let
\[
 p_0=1-\varepsilon,
 \qquad
 p_1=\varepsilon,
 \qquad
 k_m=\left\lceil\frac{m+1}{1+q}\right\rceil.
\]
Here $k_m$ is the minimum number of discovery-side observations needed for the largest candidate cutoff to pass.  The binomial variables below count how many nulls and alternatives fall on that side:
For $h\in\{0,1\}$, let
\[
 X_{0,h}\sim\operatorname{Bin}(n_0,p_h),
 \qquad
 X_1\sim\operatorname{Bin}(n_1,r),
\]
independently.  The mirror threshold satisfies
\begin{align}
 \FDR
 &\ge
 \frac12\sum_{h=0}^1
 \E\left[
  \frac{X_{0,h}}{(X_{0,h}+X_1)\vee1}
  \ind\{X_{0,h}+X_1\ge k_m\}
 \right],
 \label{eq:mixed-fdr-bound}\\
 \Power
 &\ge
 \frac12\sum_{h=0}^1
 \E\left[
  \frac{X_1}{n_1}
  \ind\{X_{0,h}+X_1\ge k_m\}
 \right].
 \label{eq:mixed-power-bound}
\end{align}
\end{proposition}

\begin{proof}
The proof has two parts.  First, adding the independent alternatives preserves PRDS: their law does not depend on $H$, so the latent-state monotonicity argument in the proof of Proposition~\ref{prop:nongaussian} applies unchanged.  The claims about the margins and the joint density are immediate.

Second, we evaluate one candidate cutoff.  Conditional on $H=h$, $X_{0,h}$ counts null $p$-values below $1/2$, and $X_1$ counts alternatives below $1/2$.  Let
\[
 u_*=\max_{1\le i\le m}\min(P_i,1-P_i).
\]
This is the largest candidate cutoff.  Every coordinate then lies in exactly one of the two counted tails, so
\[
 R(u_*)=X_{0,h}+X_1,
 \qquad
 L(u_*)=m-X_{0,h}-X_1.
\]
The cutoff $u_*$ passes exactly when
\[
 \frac{1+m-X_{0,h}-X_1}{(X_{0,h}+X_1)\vee1}\le q,
\]
which is equivalent to $X_{0,h}+X_1\ge k_m$.  When this happens, $u_*$ is the largest possible candidate and therefore equals $\widehat u$.  The event implies $X_{0,h}+X_1\ge k_m\ge1$, so the procedure rejects all discovery-side coordinates and its FDP is $X_{0,h}/(X_{0,h}+X_1)$, while its realized power is $X_1/n_1$.  Keeping only this event and averaging over the two values of $H$ proves \eqref{eq:mixed-fdr-bound}--\eqref{eq:mixed-power-bound}.
\end{proof}

\begin{corollary}[Full-cutoff population crossing]
\label{cor:mixed-phase}
Suppose $n_0/m\to\pi_0\in(0,1)$ and $n_1/m\to\pi_1=1-\pi_0$.  Define
\[
 A_h=\pi_0p_h+\pi_1r,
 \qquad
 d_h=\frac{\pi_0p_h}{A_h},
 \qquad
 c_q=\frac{1}{1+q}.
\]
Thus $A_h$ is the limiting discovery-side fraction in state $h$, $d_h$ is the limiting null fraction among those discovery-side observations, and $c_q$ is the fraction required for the largest cutoff to pass.  Then
\begin{align}
 \liminf_{m\to\infty}\FDR
 &\ge
 \frac12\sum_{h=0}^1
 d_h\ind\{A_h>c_q\},
 \label{eq:mixed-phase-fdr}\\
 \liminf_{m\to\infty}\Power
 &\ge
 \frac r2\sum_{h=0}^1
 \ind\{A_h>c_q\}.
 \label{eq:mixed-phase-power}
\end{align}
No assertion is made at the boundary $A_h=c_q$.
\end{corollary}

\begin{proof}
Conditional on $H=h$, the law of large numbers gives
\[
 \frac{X_{0,h}+X_1}{m}\longrightarrow A_h,
 \qquad
 \frac{X_{0,h}}{X_{0,h}+X_1}\longrightarrow d_h,
 \qquad
 \frac{X_1}{n_1}\longrightarrow r
\]
in probability.  Also $k_m/m\to c_q$.  If $A_h>c_q$, the full cutoff passes with probability tending to one and the displayed FDP and power limits apply.  Substitute these limits into Proposition~\ref{prop:mixed-prds} and use bounded convergence.
\end{proof}

\begin{corollary}[A large fraction of alternatives]
\label{cor:many-alternatives}
Fix $q\in(0,1/2)$ and a limiting null fraction $\pi_0\in(q/(1+q),1)$.  For every $\eta>0$, there are full-support PRDS models of the form above, with $n_0/m\to\pi_0$, for which
\[
 \liminf_{m\to\infty}\FDR\ge\frac{\pi_0}{2}-\eta,
 \qquad
 \liminf_{m\to\infty}\Power\ge\frac12-\eta.
\]
If, in addition, $\pi_0>2q$, then choosing $\eta<\pi_0/2-q$ shows that the mirror threshold fails to control the FDR even though the alternative fraction tends to $1-\pi_0$ and the power stays bounded away from zero.
\end{corollary}

\begin{proof}
Let $\varepsilon\downarrow0$ and $r\uparrow1$.  Then
\[
 A_0\longrightarrow1,
 \qquad
 A_1\longrightarrow1-\pi_0,
 \qquad
 d_0\longrightarrow\pi_0.
\]
The condition $\pi_0>q/(1+q)$ is equivalent to $1-\pi_0<c_q$.  Hence, for $\varepsilon$ sufficiently small and $r$ sufficiently close to one, only the state $H=0$ satisfies $A_h>c_q$ in Corollary~\ref{cor:mixed-phase}.  The two lower bounds can then be made arbitrarily close to $\pi_0/2$ and $1/2$, respectively.  The stronger condition $\pi_0>2q$ is needed only when one wants the FDR lower bound to exceed $q$.
\end{proof}

For example, at $q=0.1$, $\varepsilon=0.05$, and $r=0.99$, the exact binomial expressions in Proposition~\ref{prop:mixed-prds} give
\[
 \begin{array}{c|c|c|c|c}
 m&n_0&n_1&\text{FDR lower bound}&\text{power lower bound}\\ \hline
 100&90&10&0.429882&0.475022\\
 200&180&20&0.445731&0.492397
 \end{array}
\]
The favorable latent state crosses the required count more reliably as $m$ grows.  Simulating the full adaptive rule for $200{,}000$ repetitions gives FDR $0.4440$ and power $0.4864$ at $m=100$, and FDR $0.4462$ and power $0.4923$ at $m=200$.  Thus the simple full-cutoff event accounts for most of the observed FDR.

\subsection{An exact first-passage formula at reciprocal levels}
\label{subsec:first-passage}

The preceding argument is a dense one: a positive fraction of all hypotheses makes the largest cutoff pass.  The threshold can also reject for a much more local reason.  Order the hypotheses from largest magnitude to smallest.  At a reciprocal level $q=1/L$, record $+1$ for each discovery-side sign and $-L$ for each control-side sign.  A prefix passes when this running total is at least $L$; because every upward step has size one, the first passing prefix hits $L$ exactly.  Rejection is therefore a first-passage event for a simple random walk, even when no population tail ratio crosses.

\begin{proposition}[Random-design first-passage formula]
\label{prop:first-passage}
Let $q=1/L$ for an integer $L\ge2$.  Before generating the $p$-values, independently declare each coordinate null with probability $\pi_0\in(0,1)$ and alternative with probability $\pi_1=1-\pi_0$.  Conditional on the resulting label vector $J$, generate the null and alternative $p$-values as in Proposition~\ref{prop:mixed-prds}.  Write $\FDR_J$ for the ordinary FDR conditional on the fixed label vector, and let $\E_J$ denote expectation over the random labels.  Put
\[
 A_h=\pi_0p_h+\pi_1r,
 \qquad
 d_h=\frac{\pi_0p_h}{A_h}.
\]
Under random labels, $A_h$ is the discovery-side probability in state $h$, and $d_h$ is the conditional null probability among discovery-side signs.  Define
\begin{equation}
 G_{m,L}(a)
 =\sum_{\substack{j\ge1\\(L+1)j-1\le m}}
 \frac{L}{(L+1)j-1}
 \binom{(L+1)j-1}{j-1}
 a^{Lj}(1-a)^{j-1}.
 \label{eq:first-passage-prob}
\end{equation}
The quantity $G_{m,L}(a)$ is the probability that a walk with step $+1$ of probability $a$ and step $-L$ otherwise reaches $L$ within its first $m$ steps.  The cycle-lemma sum makes this probability explicit.  Then
\begin{equation}
 \E_J(\FDR_J)
 =\frac12\sum_{h=0}^1d_hG_{m,L}(A_h).
 \label{eq:random-design-fdr}
\end{equation}
Consequently, at least one deterministic null--alternative configuration has FDR no smaller than the right-hand side of \eqref{eq:random-design-fdr}.
\end{proposition}

\begin{deferredproof}{A}{Proof of Proposition~\ref{prop:first-passage}}{app:proof-first-passage}
For both the null density $f_h$ and the alternative density $g_r$, set $U_i=\min(P_i,1-P_i)$.  This variable is uniform on $(0,1/2)$ and independent of the indicator $S_i=\ind\{P_i<1/2\}$.  Ordering by decreasing $|W_i|=1/2-U_i$ is the same as ordering by increasing $U_i$.  Conditional on $H=h$, the independent random labels therefore give
\[
 \Prb(S_i=1\mid H=h)=A_h,
 \qquad
 \Prb(i\in\Hnull\mid S_i=1,H=h)=d_h.
\]
Because this ordering depends only on the $U_i$, the ordered signs are conditionally independent and identically distributed Bernoulli$(A_h)$ variables.

Let $D_k$ and $C_k$ be the discovery- and control-side sign counts among the first $k$ ordered coordinates, and put $Y_k=D_k-LC_k$.  At that prefix the knockoff+ inequality with $q=1/L$ is equivalent to
\[
 Y_k\ge L.
\]
Thus rejection occurs exactly when the random walk with increments $+1$ on a positive sign and $-L$ on a negative sign first reaches $L$ by time $m$.  If the first-passage prefix contains $j-1$ negative increments, it contains $Lj$ positive increments and has length $(L+1)j-1$.  The cycle lemma \citep{raney-1960} gives
\[
 \frac{L}{(L+1)j-1}
 \binom{(L+1)j-1}{j-1}
\]
first-passage sign sequences of this composition.  Summing their probabilities gives \eqref{eq:first-passage-prob}.  This prefix is used only to count the rejection event.  Because $\widehat t$ is the smallest passing threshold, the actual rejection set corresponds to the last passing prefix and can be larger.

Conditional on the ordered sign sequence, the selected set is determined without using the null--alternative labels.  Every selected positive sign is null with conditional probability $d_h$, independently of the other selected positive signs.  Hence the conditional expected FDP equals $d_h$ on the rejection event and zero otherwise.  Averaging over the signs and then over the two latent states proves \eqref{eq:random-design-fdr}.  The last assertion follows because at least one deterministic label vector has FDR no smaller than their average.
\end{deferredproof}

\begin{corollary}[Dense versus sparse regimes]
\label{cor:first-passage-phase}
For $a\in(0,1)$, let $G_{\infty,L}(a)=\lim_{m\to\infty}G_{m,L}(a)$.  Then
\[
 G_{\infty,L}(a)=
 \begin{cases}
  1,&a\ge L/(L+1),\\
  \zeta(a)^{-L},&a<L/(L+1),
 \end{cases}
\]
where, in the second case, $\zeta(a)>1$ is the unique nontrivial root of
\[
 a\zeta+(1-a)\zeta^{-L}=1.
\]
Therefore a state with $A_h>c_q=L/(L+1)$ yields a dense population crossing, whereas a state below that boundary still contributes the strictly positive sparse first-passage probability $\zeta(A_h)^{-L}$.  Moreover, there is a sequence of deterministic label vectors $J_m$ with $|\Hnull|/m\to\pi_0$ such that
\[
 \liminf_{m\to\infty}\FDR_{J_m}
 \ge
 \frac12\sum_{h=0}^1d_hG_{\infty,L}(A_h).
\]
\end{corollary}

\begin{deferredproof}{A}{Proof of Corollary~\ref{cor:first-passage-phase}}{app:proof-first-passage-phase}
Let $Y_k$ be the walk that steps $+1$ with probability $a$ and $-L$ otherwise, as in the definition of $G_{m,L}(a)$.  Its mean increment is $(L+1)a-L$.  At positive drift it tends to $+\infty$ almost surely.  At zero drift, the finite-variance mean-zero random walk oscillates, so its limsup is $+\infty$ almost surely.  In either case it reaches level $L$: its upward step has size one and therefore cannot jump over that level.

Now suppose the drift is negative and let $\tau_L=\inf\{k:Y_k=L\}$.  Then $Y_k\to-\infty$ almost surely.  For $f(x)=ax+(1-a)x^{-L}$, we have $f(1)=1$, $f'(1)<0$, $f''(x)>0$ for $x>0$, and $f(x)\to\infty$ as $x\to\infty$.  Hence $f(x)=1$ has exactly one root $\zeta>1$.  For this root, the process $M_k=\zeta^{Y_k}$ is a nonnegative martingale.  Before $\tau_L$, the walk is at most $L-1$, and at $\tau_L$ it equals $L$, so $M_{\tau_L\wedge n}\le\zeta^L$.  Optional stopping at the bounded time $\tau_L\wedge n$ and dominated convergence give
\[
 1=\E M_{\tau_L\wedge n}
 \longrightarrow \zeta^L\Prb(\tau_L<\infty),
\]
because $M_n\to0$ on $\{\tau_L=\infty\}$.  Hence $\Prb(\tau_L<\infty)=\zeta^{-L}$.

For the deterministic sequence, let $B_m$ denote the right-hand side of \eqref{eq:random-design-fdr}.  Then $B_m\to B_\infty:=\frac12\sum_hd_hG_{\infty,L}(A_h)$.  Let $N_{0,m}=|\Hnull|$ be the random number of null labels.  Then $N_{0,m}/m\to\pi_0$ in probability.  Choose numbers $\delta_m\downarrow0$ so slowly that
\[
 T_m=\{|N_{0,m}/m-\pi_0|\le\delta_m\}
 \quad\text{satisfies}\quad
 \Prb(T_m)\to1.
\]
Since $0\le\FDR_J\le1$,
\[
 B_m
 \le
 \Prb(T_m)\sup_{J\in T_m}\FDR_J
 +\Prb(T_m^c).
\]
Choose a deterministic $J_m\in T_m$ whose FDR is within $1/m$ of the displayed supremum.  The last two displays give $|\Hnull|/m\to\pi_0$ and $\liminf_m\FDR_{J_m}\ge B_\infty$.
\end{deferredproof}

The exact formula above is specific to reciprocal levels $q=1/L$ and the randomized-label experiment.  Its deterministic conclusion follows by averaging over the random labels and selecting one fixed configuration; it is not an exact formula for every fixed configuration.
\end{revision}

\section{Fixed Gaussian correlation}
\label{sec:fixedrho}

One could obtain a Gaussian counterexample simply by letting the correlation tend to one.  The next result is considerably stronger: at every fixed positive equicorrelation, it gives the exact limiting FDR and identifies which realizations of the common factor lead to rejection.  \rev{Equicorrelated Gaussian models also provide a classical setting for studying the asymptotic FDR of the Benjamini--Hochberg procedure \citep{finner-dickhaus-roters-2007}.}

Let $Z_0,Z_1,Z_2,\ldots$ be independent standard normal variables.  Fix $\rho\in(0,1)$ and, for each $m$, define
\begin{equation}
 W_i^{(m)}=\sqrt\rho\,Z_0+\sqrt{1-\rho}\,Z_i,
 \qquad i=1,\ldots,m.
 \label{eq:equicorr}
\end{equation}
\rev{Let $I_m$ denote the $m\times m$ identity matrix and $\mathbf1$ the all-ones vector.}  Then $W^{(m)}$ is standard Gaussian with equicorrelation $\rho$ and positive-definite covariance
\[
 \Sigma_m=(1-\rho)I_m+\rho\mathbf1\mathbf1^{\mathsf T}.
\]

Conditioning on the common factor $Z_0=z$ makes the scores independent with the $N(\sqrt\rho\,z,1-\rho)$ law.  Write $s=\sqrt{1-\rho}$, let $\Phi$ be the standard normal distribution function, and let $\barPhi=1-\Phi$.  For a threshold $t>0$, the conditional upper- and lower-tail probabilities of a single score are
\[
 \begin{aligned}
 a_z(t)&=\Prb(W_i^{(m)}\ge t\mid Z_0=z)=\barPhi\!\left(\frac{t-\sqrt\rho\,z}{s}\right),\\
 b_z(t)&=\Prb(W_i^{(m)}\le-t\mid Z_0=z)=\barPhi\!\left(\frac{t+\sqrt\rho\,z}{s}\right).
 \end{aligned}
\]

\begin{theorem}[Exact fixed-equicorrelation limit]
\label{thm:fixedrho}
Fix $q\in(0,1)$ and $\rho\in(0,1)$.  Apply the knockoff+ threshold to the all-null scores in \eqref{eq:equicorr}, and let $\widehat\Rset_m$ and $\FDR_m$ denote its rejection set and FDR.  Then
\begin{equation}
 \Prb(\widehat\Rset_m\ne\varnothing\mid Z_0=z)
 \longrightarrow \ind\{z>0\},
 \qquad z\ne0.
 \label{eq:conditional-dichotomy}
\end{equation}
Consequently,
\[
 \FDR_m\longrightarrow\frac12.
\]
In particular, if $q<1/2$, FDR control fails for all sufficiently large $m$.
\end{theorem}

\noindent\emph{Proof idea.}
Condition on $Z_0=z$.  If $z>0$, standard Gaussian tail bounds give $b_z(t)/a_z(t)\to0$ as $t\to\infty$, so a fixed threshold passes with probability tending to one by the law of large numbers.  If $z<0$, put $d=-\sqrt\rho\,z>0$.  For $X\sim N(-d,s^2)$, the conditional probability of a positive sign at magnitude $u$ is
\begin{equation}
 p_z(u)=\Prb(X>0\mid |X|=u)
 =\{1+\exp(2du/s^2)\}^{-1},
 \label{eq:negative-factor-sign-probability}
\end{equation}
which decreases to zero.  Choose $T$ so that $p_z(T)$ is below the passing fraction $c_q=1/(1+q)$.  For thresholds above $T$, conditional on the ordered magnitudes, the signs are independent Bernoulli variables; a Chernoff union bound controls every passing prefix, with an upper bound that tends to zero as $T\to\infty$.  On $[0,T]$, uniform empirical convergence keeps the positive fraction below $1/2<c_q$.  Thus the conditional rejection probability tends to $\ind\{z>0\}$; dominated convergence over $Z_0$ gives the limit $1/2$.  Appendix~\ref{app:proof-fixedrho} gives the complete proof.

\begin{storedproof}{B}{Proof of Theorem~\ref{thm:fixedrho} and the rate calculation}{app:proof-fixedrho}
Use $a_z(t)$, $b_z(t)$, and $s=\sqrt{1-\rho}$ as defined immediately before Theorem~\ref{thm:fixedrho}.  First fix $z>0$.  Writing $\phi$ for the standard normal density, the standard Mills inequalities
\[
 \frac{x}{1+x^2}\,\phi(x)\le\barPhi(x)\le\frac{\phi(x)}{x},
 \qquad x>0,
\]
show that $b_z(t)/a_z(t)\to0$ as $t\to\infty$.  Choose a fixed threshold $t_z$ for which this ratio is below $q$.  Conditional on $Z_0=z$, the law of large numbers gives
\[
 \frac{1+N_-(t_z)}{N_+(t_z)\vee1}
 \longrightarrow \frac{b_z(t_z)}{a_z(t_z)}<q
 \quad\text{almost surely}.
\]
The counts are constant between consecutive observed magnitudes.  On the event that $t_z$ passes, the smallest observed magnitude at least $t_z$ is therefore a passing element of $\mathcal T$.  The procedure rejects with conditional probability tending to one.

Now fix $z<0$ and put $d=-\sqrt\rho\,z>0$.  For one conditional score $X\sim N(-d,s^2)$, let $U=|X|$ and $S=\ind\{X>0\}$.  Equation~\eqref{eq:negative-factor-sign-probability} gives $p_z(u):=\Prb(S=1\mid U=u)$, which decreases to zero.  Order the $m$ magnitudes from largest to smallest, and let $D_k$ be the number of positive signs among the first $k$.  A prefix can pass only if
\[
 D_k\ge\frac{k+1}{1+q}>c_qk,
 \qquad c_q=\frac1{1+q}>\frac12.
\]

Fix $T>0$, to be sent to infinity after taking the limit in $m$, and let $K_T$ be the number of magnitudes at least $T$.  Because $p_z(T)<1/2<c_q$, the required upper-tail inequality is in the Chernoff regime.  The magnitude distribution is continuous, so ties have probability zero; conditional on the ordered magnitudes, their associated signs remain independent Bernoulli variables with success probabilities $p_z(U_{(k)})$.  Among the first $K_T$, these probabilities are at most $p_z(T)$.  For $c>p$, write
\[
 D_{\rm Ber}(c\|p)
 =c\log\frac cp+(1-c)\log\frac{1-c}{1-p}
\]
for the binary relative entropy.  Stochastic domination by a binomial variable and a Chernoff bound give
\[
 \begin{aligned}
 &\Prb\{D_k\ge c_qk\text{ for some }k\le K_T
       \mid U_{(1)},\ldots,U_{(m)}\}\\
 &\qquad\le
 \sum_{k\ge1}\exp\{-kD_{\rm Ber}(c_q\|p_z(T))\}.
 \end{aligned}
\]
This geometric upper bound tends to zero as $T\to\infty$ because $p_z(T)\to0$.

It remains to rule out bounded thresholds.  The Glivenko--Cantelli theorem applied to the two half-line indicator classes gives
\[
 \sup_{0\le t\le T}
 \left|
 \frac{N_+(t)}{N_+(t)+N_-(t)}
 -\frac{\Prb(X\ge t)}{\Prb(|X|\ge t)}
 \right|\longrightarrow0
 \quad\text{in probability}.
\]
Here the denominator is uniformly well behaved because $\Prb(|X|\ge t)\ge\Prb(|X|\ge T)>0$ on $[0,T]$.  Since $X\sim N(-d,s^2)$ with $d>0$, $\Prb(X\ge t)<\Prb(X\le-t)$ for every $t\ge0$.  The limiting positive fraction is therefore continuous and strictly below $1/2$ throughout $[0,T]$, and hence uniformly below $c_q$.  Because a passing threshold requires an empirical positive fraction greater than $c_q$, the probability that any threshold in $[0,T]$ passes tends to zero.  Taking the upper limit in $m$ and then letting $T\to\infty$ proves that the conditional rejection probability tends to zero for $z<0$.

This establishes \eqref{eq:conditional-dichotomy}.  Under the global null, the FDR is the rejection probability.  Dominated convergence over $Z_0$ therefore gives $\FDR_m\to\Prb(Z_0>0)=1/2$.

\par\medskip\noindent\emph{Quantitative lower bound.}\label{app:rate-fixedrho}
Write $s=\sqrt{1-\rho}$ and, for large $m$, set
\[
 t_m=\sqrt{2(1-\rho)(\log m-\log\log m)},
 \qquad
 \delta_m=\frac{\sqrt{1-\rho}\,\log(2/q)}{2\sqrt{2\rho(\log m-\log\log m)}}.
\]
At $z=\delta_m$ the standardized tail arguments $x_m=(t_m-\sqrt\rho\,\delta_m)/s$ and $y_m=(t_m+\sqrt\rho\,\delta_m)/s$ satisfy $\tfrac12(y_m^2-x_m^2)=\log(2/q)$, so the Mills bounds above give
\[
 \frac{b_{\delta_m}(t_m)}{a_{\delta_m}(t_m)}
 \le\{1+o(1)\}\,\frac q2,
\]
and the same monotonicity in $z$ extends the bound to every $z\ge\delta_m$.  The scale $t_m$ is calibrated so that the exceedance count still diverges: $\inf_{z\ge\delta_m}m\,a_z(t_m)=m\,a_{\delta_m}(t_m)$ is of order $\sqrt{\log m}$ and hence tends to infinity.
A Chernoff bound therefore shows that, uniformly over $z\ge\delta_m$, the empirical ratio $(1+N_-(t_m))/(N_+(t_m)\vee1)$ is at most $q$ with probability $1-\exp\{-c\sqrt{\log m}\}$ for some $c>0$.  Integrating over the common factor,
\[
 \FDR_m
 \ge\barPhi(\delta_m)-o\!\left((\log m)^{-1/2}\right)
 =\frac12-\frac{\sqrt{1-\rho}\,\log(2/q)}{4\sqrt{\pi\rho\log m}}
 +o\!\left((\log m)^{-1/2}\right).
\]
\end{storedproof}

The conditional phase transition is governed by the shared factor rather than by the size of one pairwise correlation.  The result applies, for example, at $\rho=0.01$; when many coordinates share one factor, small pairwise correlation need not mean weak overall dependence.

\begin{remark}[A quantitative lower bound]
\label{rem:rate}
A quantitative version of the positive-factor argument gives
\[
 \FDR_m
 \ge\frac12-\frac{\sqrt{1-\rho}\,\log(2/q)}{4\sqrt{\pi\rho\log m}}
 +o\!\left((\log m)^{-1/2}\right).
\]
Appendix~\ref{app:rate-fixedrho} gives the derivation.  The gap between this lower bound and one half is of order $(\log m)^{-1/2}$, with a constant that grows as $\rho$ decreases.  This slowly decaying one-sided guarantee is consistent with the finite-$m$ behavior in Figure~\ref{fig:numerics}(b).  Theorem~\ref{thm:fixedrho} establishes convergence to $1/2$; the calculation does not give a matching upper rate.
\end{remark}

\begin{corollary}[PRDS still does not suffice]
\label{cor:prds}
Let $P_i^{(m)}=1-\Phi(W_i^{(m)})$.  Each $P_i^{(m)}$ is uniform on $(0,1)$, and the vector is PRDS.  For every fixed $\rho>0$, the mirror threshold has FDR converging to $1/2$.  Hence, for every $q<1/2$, its FDR is greater than $q$ for all sufficiently large $m$.
\end{corollary}

\begin{proof}
Uniformity and equivalence to the signed-score rule follow from Proposition~\ref{prop:equivalence}.  The PRDS assertion is a well-known result for one-sided $p$-values from a multivariate normal vector with nonnegative correlations; see \citet{benjamini-yekutieli-2001} and \citet{sarkar-2002}.  Since every off-diagonal entry of $\Sigma_m$ equals $\rho>0$, that result applies.  The false discovery statement follows from Theorem~\ref{thm:fixedrho}.
\end{proof}

\rev{A direct finite example is also useful.  Take $q=0.1$, $m=11$, and $\rho=0.8$.  On the event}
\[
 Z_0>1,
 \qquad Z_i>-2\quad(i=1,\ldots,11),
\]
all scores are positive because $\sqrt{0.8}-2\sqrt{0.2}=0$.  Lemma~\ref{lem:trigger} gives
\[
 \FDR\ge \barPhi(1)\Phi(2)^{11}\approx0.12317>0.1.
\]
The covariance is positive definite, the joint density is smooth and positive, and the scores are almost surely distinct.

\begin{revision}
The equicorrelation model has one important geometric feature: every null score loads on the common factor in the same direction.  When the factor is negative, all null signs are conditionally biased toward the control side.  This protects half of the factor realizations.  The next result formalizes that observation and allows the nonnegative loadings to differ across nulls.

\begin{theorem}[A ceiling under one-sign Gaussian loadings]
\label{thm:positive-loading-upper}
Fix $q\in(0,1)$.  Let $\Hnull$ be any set of true nulls and, for $i\in\Hnull$, let $\lambda_i\in[0,1)$.  Suppose that
\[
 W_i=\lambda_iZ_0+\sqrt{1-\lambda_i^2}\,Z_i,
\]
where $Z_0$ and $(Z_i)_{i\in\Hnull}$ are independent standard normal variables.  The nonnull score vector $(W_j)_{j\notin\Hnull}$ may otherwise be arbitrary and may depend on $Z_0$.  We require one separation condition: given $Z_0$, the nonnull vector is independent of the null residuals $(Z_i)_{i\in\Hnull}$.  Then the knockoff+ threshold satisfies the finite-sample bound
\begin{equation}
 \FDR\le\frac{1+q}{2}.
 \label{eq:positive-loading-upper}
\end{equation}
\end{theorem}

\begin{proof}
Condition on $Z_0=z$.  For a null coordinate $i$, let $f_{i,z}$ denote the conditional density, which is that of $N(\lambda_i z,1-\lambda_i^2)$.  Then, for $u>0$,
\[
 \frac{\Prb(W_i>0\mid |W_i|=u,Z_0=z)}
      {\Prb(W_i<0\mid |W_i|=u,Z_0=z)}
 =\frac{f_{i,z}(u)}{f_{i,z}(-u)}
 =\exp\left\{\frac{2u\lambda_i z}{1-\lambda_i^2}\right\}.
\]
If $z\le0$, this ratio is at most one.  The conditional independence assumption implies that, after conditioning further on the null magnitudes and on all nonnull scores, the null signs remain independent and are biased toward the control side.  Proposition~\ref{prop:conservative-signs}, applied under the conditional law given $Z_0=z$, therefore gives
\[
 \E(\FDP\mid Z_0=z)\le q,
 \qquad z\le0.
\]
For $z>0$, the trivial bound $\FDP\le1$ applies.  Since $Z_0$ is symmetric,
\[
 \FDR
 \le q\Prb(Z_0\le0)+\Prb(Z_0>0)
 =\frac{1+q}{2}.
\]
\end{proof}

Theorems~\ref{thm:fixedrho} and \ref{thm:positive-loading-upper} locate the severity of failure in the equal-loading subclass $\lambda_i=\sqrt\rho$.  For each $m$, let $\mathcal M_m(\rho)$ be the class of joint score distributions satisfying the assumptions of Theorem~\ref{thm:positive-loading-upper}, with $\lambda_i=\sqrt\rho$ for every null $i$.  The all-null sequence in Theorem~\ref{thm:fixedrho} has FDR converging exactly to $1/2$, while Theorem~\ref{thm:positive-loading-upper} gives the upper bound for every member.  Hence
\[
 \frac12
 \le
 \liminf_{m\to\infty}
 \sup_{\mathbb P\in\mathcal M_m(\rho)}\FDR_m(\mathbb P)
 \le
 \limsup_{m\to\infty}
 \sup_{\mathbb P\in\mathcal M_m(\rho)}\FDR_m(\mathbb P)
 \le\frac{1+q}{2}.
\]

The proof also suggests a broader principle.  On any factor event where every null conditional mean is nonpositive, the conditional FDR is at most $q$; elsewhere it is at most one.
\end{revision}

\section{FDR arbitrarily close to one with identical null margins}
\label{sec:nearone}

\begin{revision}
The one-sign ceiling raises a natural question: can a different loading pattern make the failure substantially worse?  Yes.  We first use opposing Gaussian loadings to obtain FDR and power near one.  We then give a non-Gaussian construction in which higher-order sign alignment is invisible to every pairwise covariance.
\end{revision}

\begin{revision}
\begin{theorem}[Standardized Gaussian failure with alternatives and small correlations]
\label{thm:gaussian-nearone}
Fix $q\in(0,1)$, $\bar\rho\in(0,1)$, and $\eta>0$.  There are a constant $\mu>0$, a constant $\gamma\in(0,1)$, and a sequence of jointly Gaussian score vectors with the following properties:
\begin{enumerate}
 \item every null score has the exact $N(0,1)$ distribution;
 \item the fraction of nonnulls tends to $\gamma$, and every nonnull score has distribution $N(\mu,1)$;
 \item every covariance matrix is positive definite and
 \[
  \max_{i\ne j}|\operatorname{Corr}(W_i,W_j)|<\bar\rho;
 \]
 \item if $\FDR_m$ and $\Power_m$ denote the threshold's FDR and power, then
 \[
  \liminf_{m\to\infty}\FDR_m>1-\eta,
  \qquad
  \liminf_{m\to\infty}\Power_m>1-\eta.
 \]
\end{enumerate}
Moreover, the transformation $P_i=1-\Phi(W_i)$ gives exactly uniform null $p$-values, stochastically smaller nonnull $p$-values, and the same rejection sets under the mirror threshold.
\end{theorem}

\emph{Proof idea.}
For $0<c<a<1$, let $Z,\varepsilon_1,\varepsilon_2,\ldots$ be independent standard normal variables.  Use two null blocks
\[
 W_i=-aZ+\sqrt{1-a^2}\,\varepsilon_i
 \quad\text{and}\quad
 W_i=cZ+\sqrt{1-c^2}\,\varepsilon_i,
\]
and independent alternatives $W_i=\mu+\varepsilon_i$.  When $Z<0$, the first block's mean shift produces the needed separation, and a small second-block proportion keeps its controls negligible.  When $Z>0$, the second block favors the discovery side, and its larger residual variance makes that favorable tail decay more slowly than the relevant control tails.  We then choose the block and alternative proportions so that the empirical ratio passes but most selected variables are null.  Taking $\mu$ large gives power near one at the same time.

\begin{deferredproof}{B}{Proof of Theorem~\ref{thm:gaussian-nearone}}{app:proof-gaussian-nearone}
The parameters are chosen in stages.  We first isolate a high-probability range of the common factor, then obtain tail separation on its negative and positive parts, and only afterward choose the alternative fraction and mean.  This order avoids any circular dependence among the constants.

\emph{Factor range and loadings.}
It is enough to consider $\eta\in(0,1)$.  Choose $\beta>0$ so small that
\begin{equation}
 (1-\beta)^2>1-\eta.
 \label{eq:beta-nearone}
\end{equation}
Let $Z$ be standard normal, and choose $0<\delta<M<\infty$ such that
\begin{equation}
 \Prb(\delta\le |Z|\le M)>1-\beta.
 \label{eq:factor-compact-event}
\end{equation}
Choose
\[
 0<c<a<\min\{1,\sqrt{\bar\rho}\},
 \qquad
 s_a=\sqrt{1-a^2},
 \qquad
 s_c=\sqrt{1-c^2}.
\]
The two null blocks will have factor loadings $-a$ and $c$.  Conditional on $Z=z$, define their positive- and negative-tail probabilities by
\begin{align*}
 A_z^+(t)&=\barPhi\left(\frac{t+az}{s_a}\right),
 &A_z^-(t)&=\barPhi\left(\frac{t-az}{s_a}\right),\\
 B_z^+(t)&=\barPhi\left(\frac{t-cz}{s_c}\right),
 &B_z^-(t)&=\barPhi\left(\frac{t+cz}{s_c}\right).
\end{align*}

\emph{Tail separation.}
We first choose a mixing proportion $\pi\in(0,1)$ for the second null block and two deterministic thresholds.  Uniformly for $z\in[-M,-\delta]$,
\begin{equation}
 \frac{A_z^-(t)}{A_z^+(t)}\longrightarrow0
 \qquad(t\to\infty).
 \label{eq:negative-factor-tail-ratio}
\end{equation}
This follows directly from standard Gaussian tail bounds (Mills-ratio bounds): the two standardized arguments have the same variance scale, while their squared difference grows at least linearly in $t\delta$.  Choose $t_->0$ so that the supremum in \eqref{eq:negative-factor-tail-ratio} is smaller than $q/8$, and put
\[
 d_-:=\inf_{-M\le z\le-\delta}A_z^+(t_-)>0.
\]
Now choose $\pi>0$ sufficiently small that
\begin{equation}
 \frac{\pi}{(1-\pi)d_-}<\frac q8.
 \label{eq:choose-pi}
\end{equation}

For $z\in[\delta,M]$, the second block eventually dominates both sources of controls.  More precisely, because $s_c>s_a$, the Mills bounds give, uniformly on this compact interval,
\begin{equation}
 \frac{B_z^-(t)}{B_z^+(t)}\longrightarrow0,
 \qquad
 \frac{A_z^-(t)}{B_z^+(t)}\longrightarrow0
 \qquad(t\to\infty).
 \label{eq:positive-factor-tail-ratios}
\end{equation}
The first limit uses the positive mean $cz\ge c\delta$ of the second block.  In the second, the larger residual variance $s_c^2$ makes the positive tail of the second block decay more slowly than the negative tail of the first block.  Choose $t_+>0$ so large that
\[
 \sup_{\delta\le z\le M}\frac{B_z^-(t_+)}{B_z^+(t_+)}<\frac q8,
 \qquad
 \sup_{\delta\le z\le M}\frac{A_z^-(t_+)}{B_z^+(t_+)}
 <\frac{q\pi}{8(1-\pi)}.
\]

\emph{Combine the null blocks.}
For a null drawn from the two blocks in proportions $1-\pi$ and $\pi$, write
\[
 D_z(t)=(1-\pi)A_z^+(t)+\pi B_z^+(t),
 \qquad
 C_z(t)=(1-\pi)A_z^-(t)+\pi B_z^-(t).
\]
The choices above imply
\begin{equation}
 \frac{C_z(t_-)}{D_z(t_-)}<\frac q4
 \quad(-M\le z\le-\delta),
 \qquad
 \frac{C_z(t_+)}{D_z(t_+)}<\frac q4
 \quad(\delta\le z\le M).
 \label{eq:null-tail-ratio-both-signs}
\end{equation}
Indeed, on the negative interval,
\[
 \frac{C_z(t_-)}{D_z(t_-)}
 \le \frac{A_z^-(t_-)}{A_z^+(t_-)}
     +\frac{\pi}{(1-\pi)A_z^+(t_-)}<\frac q4,
\]
and on the positive interval,
\[
 \frac{C_z(t_+)}{D_z(t_+)}
 \le \frac{B_z^-(t_+)}{B_z^+(t_+)}
     +\frac{1-\pi}{\pi}\frac{A_z^-(t_+)}{B_z^+(t_+)}<\frac q4.
\]
By continuity and compactness,
\begin{equation}
 d_*:=\min\left
 \{
  \inf_{-M\le z\le-\delta}D_z(t_-),
  \inf_{\delta\le z\le M}D_z(t_+)
 \right\}>0.
 \label{eq:d-star}
\end{equation}

\emph{Add alternatives.}
Choose a positive alternative fraction $\gamma$ sufficiently small that
\begin{equation}
 \frac{\gamma}{(1-\gamma)d_*}
 <\min\left\{\frac q4,\frac{\beta}{1-\beta}\right\}.
 \label{eq:choose-gamma}
\end{equation}
Finally, with $T=\max(t_-,t_+)$, choose $\mu$ so large that
\begin{equation}
 \barPhi(T-\mu)>1-\beta.
 \label{eq:choose-mu}
\end{equation}

\emph{Finite vectors and a passing cutoff.}
Let $n_1^{(m)}=\lfloor\gamma m\rfloor$ be the number of alternatives and $n_0^{(m)}=m-n_1^{(m)}$ the number of nulls.  Split the nulls into blocks $A_m$ and $B_m$ with $|B_m|/n_0^{(m)}\to\pi$.  Using mutually independent standard normal variables $Z$, $(\varepsilon_i)$, set
\begin{align*}
 W_i&=-aZ+s_a\varepsilon_i,
 &&i\in A_m,\\
 W_i&= cZ+s_c\varepsilon_i,
 &&i\in B_m,\\
 W_i&=\mu+\varepsilon_i,
 &&i\in\Halt.
\end{align*}
Every null margin is $N(0,1)$ and every alternative margin is $N(\mu,1)$.  The nonzero correlations are $a^2$ within $A_m$, $c^2$ within $B_m$, and $-ac$ across the two null blocks.  Their absolute values are all smaller than $\bar\rho$.  The independent residual variances are strictly positive, so the covariance matrix is positive definite.

Fix $z$ in \eqref{eq:factor-compact-event}, taking $t_z=t_-$ for $z<0$ and $t_z=t_+$ for $z>0$.  Conditional on $Z=z$, the scores are independent, and the law of large numbers and \eqref{eq:null-tail-ratio-both-signs}--\eqref{eq:choose-gamma} bound the limiting control-to-discovery count ratio at $t_z$ by
\[
 \frac{C_z(t_z)}{D_z(t_z)}
 +\frac{\gamma}{(1-\gamma)D_z(t_z)}
 <\frac q2<q.
\]
The added one is negligible because the discovery count has order $m$.  Thus, with conditional probability tending to one, at least one positive score is at least $t_z$.  Let $t'_z$ be the smallest such score.  Then $t'_z\in\mathcal T$, $N_+(t'_z)=N_+(t_z)$, and $N_-(t'_z)\le N_-(t_z)$, so $t'_z$ passes and the selected threshold is no larger than $t'_z$.

\emph{FDR and power.}
On that event the selected set contains every null score exceeding $t_z$, while it contains at most all $n_1^{(m)}$ alternatives.  Therefore
\[
 \FDP
 \ge
 \frac{\#\{i\in\Hnull:W_i\ge t_z\}}
      {\#\{i\in\Hnull:W_i\ge t_z\}+n_1^{(m)}}.
\]
The conditional limit of the right-hand side is at least
\[
 \frac{(1-\gamma)d_*}{(1-\gamma)d_*+\gamma}>1-\beta
\]
by \eqref{eq:choose-gamma}.  The selected set also contains every alternative score exceeding $t_z$, so its power has conditional lower limit at least
\[
 \Prb\{N(\mu,1)\ge t_z\}
 \ge\barPhi(T-\mu)>1-\beta
\]
by \eqref{eq:choose-mu}.  Fatou's lemma, \eqref{eq:factor-compact-event}, and \eqref{eq:beta-nearone} now give
\[
 \liminf_{m\to\infty}\FDR_m
 \ge(1-\beta)\Prb(\delta\le|Z|\le M)
 >(1-\beta)^2>1-\eta,
\]
and the same calculation gives the power bound.

For the last assertion, $P_i=1-\Phi(W_i)$ is uniform under every null and is stochastically smaller under $N(\mu,1)$.  Also, for $t>0$ and $u=1-\Phi(t)$,
\[
 P_i\le u\Longleftrightarrow W_i\ge t,
 \qquad
 P_i\ge1-u\Longleftrightarrow W_i\le-t.
\]
Thus the score and $p$-value thresholds have the same rejection sets.
\end{deferredproof}

\begin{remark}
\label{rem:gaussian-nearone-dimension}
The proof separately bounds several contributions to the control-to-discovery ratio by $q/8$.  This can force the second null block and the alternative fraction to be extremely small, and the required dimension can be astronomical when $\bar\rho$ is small.  Figure~\ref{fig:numerics}(d) instead illustrates the mechanism at a moderate dimension and a larger maximum correlation.

Small pairwise correlations do not imply weak dependence in every sense.  At fixed block proportions, the rank-one factor component has a leading eigenvalue of order $m$.  The theorem controls individual pairwise correlations, not dependence that accumulates across many coordinates.
\end{remark}
\end{revision}

\rev{The Gaussian construction above still has a common factor with nonzero pairwise correlations; the next result shows that near-total FDR failure can persist even when every pairwise covariance is zero.}

\begin{theorem}[An exchangeable, pairwise-uncorrelated worst case]
\label{thm:nearone}
For every $q\in(0,1)$ and every $\eta>0$, there are $m$ all-null scores $W_1,\ldots,W_m$ with the following properties:
\begin{enumerate}
 \item the vector is exchangeable and has a strictly positive joint density on $\mathbb R^m$;
 \item all coordinates have the same continuous distribution, symmetric about zero;
 \item $\operatorname{Cov}(W_i,W_j)=0$ for every $i\ne j$;
 \item the knockoff+ threshold has $\FDR>1-\eta$.
\end{enumerate}
Moreover, if $F$ is the common marginal distribution function and $P_i=1-F(W_i)$, then the $P_i$ are exchangeable and exactly uniform, their joint density is positive on $(0,1)^m$, and the mirror threshold has the same FDR.
\end{theorem}

\begin{revision}
\emph{Proof idea.}
Choose $0<\pi<q/(1+q)$, let $L_i$ equal $a>0$ with probability $\pi$ and $-b<0$ otherwise, and set $b=\pi a/(1-\pi)$.  Let $Z,\varepsilon_1,\ldots,\varepsilon_m$ be independent standard normal variables and, for $\sigma>0$, define
\begin{equation}
 W_i=L_iZ+\sigma\varepsilon_i,
 \qquad i=1,\ldots,m.
 \label{eq:exchangeable-factor}
\end{equation}
Then $\E L_i=0$, which makes distinct coordinates uncorrelated, although their signs remain strongly aligned through $Z$.  If $K=\#\{i:L_i=a\}$, the law of large numbers makes both $1/K\le q$ and $(1+K)/(m-K)\le q$ overwhelmingly likely for large $m$.  These two inequalities make the threshold pass for the two signs of $Z$, respectively, as $\sigma\downarrow0$.
\end{revision}

\begin{deferredlegacyproof}{B}{Proof of Theorem~\ref{thm:nearone}}{app:proof-exchangeable-nearone}
Choose
\[
 0<\pi<\frac{q}{1+q}
\]
and a constant $a>0$.  Set $b=\pi a/(1-\pi)$, so $0<b<a$.  Let $L_1,\ldots,L_m$ be independent with
\[
 L_i=
 \begin{cases}
  a,&\text{with probability }\pi,\\
  -b,&\text{with probability }1-\pi,
 \end{cases}
\]
and let $Z,\varepsilon_1,\ldots,\varepsilon_m$ be independent standard normal variables, independent of the $L_i$.  \rev{For $\sigma>0$, use the scores in \eqref{eq:exchangeable-factor}.}

The vector is exchangeable.  Conditional on $L=(L_1,\ldots,L_m)$, it is Gaussian with covariance $LL^{\mathsf T}+\sigma^2I_m$, which is positive definite.  Its unconditional density, a finite mixture of positive Gaussian densities, is therefore positive everywhere.  Each coordinate has the same symmetric mixture law
\[
 \mathcal L(W_i)
 =\pi N(0,a^2+\sigma^2)+(1-\pi)N(0,b^2+\sigma^2),
\]
whose distribution function we denote by $F_\sigma$.
Also, $\E L_i=\pi a-(1-\pi)b=0$, and for $i\ne j$,
\[
 \operatorname{Cov}(W_i,W_j)
 =\E(L_iL_jZ^2)
 =(\E L_i)(\E L_j)=0.
\]

Let $K=\#\{i:L_i=a\}$ and define
\[
 G_m=
 \left\{
  \frac1K\le q,
  \quad
  \frac{1+K}{m-K}\le q
 \right\},
\]
with the first inequality interpreted as false when $K=0$ and the second as false when $K=m$.  By the law of large numbers, $K/m\to\pi$, while
\[
 \frac{\pi}{1-\pi}<q.
\]
Consequently, $\Prb(G_m)\to1$ as $m\to\infty$.

Fix $m$, a realization of $L$ in $G_m$, and $z>0$.  As $\sigma\downarrow0$, the scores with $L_i=a$ approach $az>0$, while those with $L_i=-b$ approach $-bz<0$.  Since $a>b$, with conditional probability tending to one the $K$ positive scores have larger magnitudes than all negative scores.  At the smallest of those $K$ positive scores, $N_+=K$ and $N_-=0$, so the first inequality defining $G_m$ makes the threshold pass.

Now fix $z<0$.  The $K$ scores with $L_i=a$ approach $-a|z|$, while the remaining $m-K$ scores approach $b|z|$.  Again $a>b$.  With conditional probability tending to one, at the smallest positive score we have $N_+=m-K$ and $N_-=K$, so the second inequality defining $G_m$ makes the threshold pass.

Thus, for every realization $L\in G_m$ and every $z\ne0$, the conditional probability of a rejection tends to one as $\sigma\downarrow0$.  \rev{Let $\FDR_{m,\sigma}$ denote the FDR for the corresponding values of $m$ and $\sigma$.}  Fatou's lemma and the global-null identity \eqref{eq:global-null} give
\[
 \liminf_{\sigma\downarrow0}\FDR_{m,\sigma}
 \ge \Prb(G_m).
\]
Choose $m$ so that $\Prb(G_m)>1-\eta/2$, and then choose $\sigma>0$ sufficiently small to obtain $\FDR_{m,\sigma}>1-\eta$.

The common marginal density is everywhere positive, so $F_\sigma$ is continuous and strictly increasing.  Therefore $P_i=1-F_\sigma(W_i)$ is exactly uniform.  Symmetry gives
\[
 P_i\le 1-F_\sigma(t)\Longleftrightarrow W_i\ge t,
 \qquad
 P_i\ge F_\sigma(t)\Longleftrightarrow W_i\le-t,
\]
so Proposition~\ref{prop:equivalence} applies after a monotone change of threshold.  Exchangeability and positivity of the joint density are preserved by the coordinatewise transformation.
\end{deferredlegacyproof}

\begin{revision}
This theorem concerns a different limitation from PRDS; Proposition~\ref{prop:nongaussian} and Corollary~\ref{cor:prds} cover that condition.  Here the dependence is hidden beyond pairwise covariance: exchangeability and identical null distributions still allow strong higher-order sign alignment.
\end{revision}

\begin{revision}
\section{Limits and calibrated repairs}
\label{sec:repair}

The counterexamples leave two questions.  Can the failure be repaired by changing only the two tail counts?  If not, what additional information restores a finite-sample guarantee?  We first rule out count-only corrections, then quantify departures from conditional sign symmetry, and finally specialize those bounds to a specified Gaussian joint model.  The last subsection explains why covariance is useful only through such a model.

\subsection{No count-only repair}
\end{revision}

\begin{revision}
A natural response is to make the mirror estimate more conservative.  For example, one might replace the plus one by a larger constant.  More generally, one could use any deterministic function of the two current counts.  The next result shows that no such monotone two-count rule can be both useful and uniformly valid over the PRDS class in Proposition~\ref{prop:nongaussian} when $q<1/2$.
\end{revision}

For each $m$, let
\[
 \psi_m:\{0,\ldots,m\}^2\longrightarrow[0,\infty]
\]
be deterministic, nonincreasing in its first argument and nondecreasing in its second.  Think of $\psi_m(r,\ell)$ as an estimated FDP based only on $r$ positive scores and $\ell$ negative scores.  Consider the threshold below, with $\widehat t_\psi=\infty$ when the set is empty and rejection set $\{i:W_i\ge\widehat t_\psi\}$:
\begin{equation}
\widehat t_\psi
=\min\left\{t\in\mathcal T:
       \psi_m\{N_+(t),N_-(t)\}\le q\right\}.
\label{eq:general-count-rule}
\end{equation}
The monotonicity directions encode the usual logic of a mirror estimate: adding positive scores should not increase the estimated FDP, whereas adding negative scores should not decrease it.  The ordinary knockoff+ threshold has $\psi_m(r,\ell)=(1+\ell)/(r\vee1)$.  Additive and multiplicative two-count modifications are included as special cases.

\begin{proposition}[No nontrivial monotone count-only correction]
\label{prop:no-repair}
Fix $q\in(0,1/2)$ and $m\ge1$.
\begin{enumerate}
 \item If $\psi_m(m,0)\le q$, then for every $c<1/2$ there is an all-null vector of exactly uniform PRDS $p$-values, with a strictly positive joint density on $(0,1)^m$, for which the rule \eqref{eq:general-count-rule} has $\FDR>c$.
 \item If $\psi_m(m,0)>q$, then the rule \eqref{eq:general-count-rule} never rejects, for any input vector.
\end{enumerate}
Consequently, every deterministic monotone rule based only on the two tail counts is either invalid over this class or identically powerless.
\end{proposition}

\begin{proof}
Suppose first that $\psi_m(m,0)\le q$.  Apply the rule to the scores $W_i=1/2-P_i$ from the construction in Proposition~\ref{prop:nongaussian}.  On an outcome where all $m$ scores are positive, take $t=\min_iW_i$.  Then $N_+(t)=m$ and $N_-(t)=0$, so the rule rejects.  In the full-support PRDS construction of Proposition~\ref{prop:nongaussian}, the probability of this event is
\[
 \frac12\{(1-\varepsilon)^m+\varepsilon^m\},
\]
which tends to $1/2$ as $\varepsilon\downarrow0$.  Under the global null, the FDR is the rejection probability, proving the first assertion.

Now suppose $\psi_m(m,0)>q$.  At every threshold, $N_+(t)\le m$ and $N_-(t)\ge0$.  Monotonicity therefore gives
\[
 \psi_m\{N_+(t),N_-(t)\}
 \ge \psi_m(m,0)>q.
\]
Thus no threshold can pass.
\end{proof}

\begin{revision}
\begin{corollary}[The count-only impossibility persists with alternatives]
\label{cor:no-repair-mixed}
Fix $q\in(0,1/2)$ and $\pi_0\in(2q,1)$.  Consider any sequence of deterministic monotone count rules of the form \eqref{eq:general-count-rule}.
\begin{enumerate}
 \item If $\psi_m(m,0)>q$ for all sufficiently large $m$, then the rules are eventually powerless on every input vector.
 \item Otherwise, there are an infinite subsequence $\mathcal M\subseteq\mathbb N$ and full-support PRDS models indexed by $m\in\mathcal M$, with $n_0/m\to\pi_0$ along $\mathcal M$ and a positive fraction of genuine alternatives, such that
 \[
  \liminf_{\substack{m\to\infty\\m\in\mathcal M}}\FDR_m
  \ge\frac{\pi_0}{2}>q,
  \qquad
  \liminf_{\substack{m\to\infty\\m\in\mathcal M}}\Power_m
  \ge\frac12.
 \]
\end{enumerate}
Thus a count-only rule that remains capable of rejecting is invalid even on mixed configurations with nonvanishing power.
\end{corollary}

\begin{proof}
The first assertion is Proposition~\ref{prop:no-repair}.  For the second, take an infinite subsequence $\mathcal M$ on which $\psi_m(m,0)\le q$, let $n_0/m\to\pi_0$ along $\mathcal M$, and put $n_1=m-n_0$.  Use the mixed model of Proposition~\ref{prop:mixed-prds} with
\[
 \varepsilon_m=(m+1)^{-2},
 \qquad
 r_m=1-(m+1)^{-2}.
\]
The joint density is positive and the vector is PRDS with respect to its null coordinates for every $m\in\mathcal M$.  On the event that $H=0$ and every coordinate lies below $1/2$, the current counts at the smallest score magnitude are $(m,0)$.  The rule therefore rejects all hypotheses.  Along $\mathcal M$, the probability of this event is
\[
 \frac12(1-\varepsilon_m)^{n_0}r_m^{n_1}\longrightarrow\frac12.
\]
Its FDP is $n_0/m$ and its power is one.  Keeping only this event gives the two subsequential lower bounds.
\end{proof}

Proposition~\ref{prop:no-repair} is deliberately limited to deterministic monotone rules that use the two counts at the current threshold.  Randomization, nonmonotone rules, and procedures using the full sequence of counts across thresholds are outside its scope.  The obstruction is nevertheless informational rather than a defect of the plus-one constant: the two current counts cannot distinguish an ordinary sign fluctuation from a dependence-induced imbalance.  We therefore take a different route and quantify how far the joint null signs are from the ideal sign-flip law.

\subsection{FDR bounds from approximate null-sign symmetry}
\label{subsec:positive}

We use two diagnostics for the departure from ideal sign flips.  The first compares the entire conditional null-sign vector with independent fair coins.  The second examines, one null coordinate at a time, how strongly the remaining scores can predict its sign.  The coordinatewise diagnostic leads to two complementary guarantees: one excludes an event of small probability, and the other averages the relevant odds over the negative scores used by the mirror estimate.

To formalize these comparisons, suppose that every null score is nonzero almost surely; zero coordinates may simply be omitted, as in Assumption~\ref{ass:signflip}.  Conditional on
\[
 \mathcal G
 =\sigma\bigl(|W_1|,\ldots,|W_m|,(W_i)_{i\notin\Hnull}\bigr),
\]
let $\nu_{\mathcal G}$ be the conditional law of the null sign vector.  Let $\nu^\circ_{\mathcal G}$ be the reference law of independent fair signs on the same coordinates.  Write $d_{\rm TV}$ for total-variation distance and $D_{\rm KL}$ for Kullback--Leibler divergence (relative entropy).  Averaging these conditional discrepancies over $\mathcal G$, define
\[
 d_{\rm sign}
 =\E\{d_{\rm TV}(\nu_{\mathcal G},\nu^\circ_{\mathcal G})\},
 \qquad
 K_{\rm sign}
 =\E\{D_{\rm KL}(\nu_{\mathcal G}\|\nu^\circ_{\mathcal G})\}.
\]
Conditioning on $\mathcal G$ freezes the score magnitudes and the scores for the genuine alternatives; only the null signs remain random.
For the coordinatewise comparison, define the conditional log-odds
\begin{equation}
 \Lambda_j
 =\log
 \frac{\Prb(W_j>0\mid |W_j|,W_{-j})}
      {\Prb(W_j<0\mid |W_j|,W_{-j})},
 \qquad
 \Lambda_{\max}=\max_{j\in\Hnull}\Lambda_j,
 \label{eq:conditional-log-odds}
\end{equation}
for each null $j$.  Positive values favor the discovery side.  The ratio compares the two possible signs of $W_j$ while holding its magnitude and every other score fixed.  Set $\Lambda_j=+\infty$ when the denominator is zero and $\Lambda_j=-\infty$ when the numerator is zero; the null score is nonzero almost surely, so both cannot vanish.  Also set $\max\varnothing=-\infty$.  If $\Hnull=\varnothing$, then $\FDR=0$ and all bounds below are immediate.
Write $\FDR(q)$ for the FDR of the threshold when it is run at nominal level $q$.

The integrated argument uses negative null scores as leave-one-out controls.  For each null $j$, it recomputes the cutoff after forcing only the sign of $W_j$ to be positive; the resulting cutoff is then determined by $(|W_j|,W_{-j})$.  This argument additionally assumes that both conditional sign probabilities in \eqref{eq:conditional-log-odds} are positive almost surely.  Under this assumption, put
$\Omega_j=e^{\Lambda_j}$.  For each null $j$, let $W^{(j)}$ be obtained from $W$ by replacing $W_j$ with $|W_j|$, and define the counterfactual threshold
\[
 \widehat t_j(q)=\widehat t(W^{(j)};q).
\]
This is an analytical device rather than a second testing procedure.  Its active negative-null set is
\begin{equation}
 \mathcal A_q
 =\{j\in\Hnull:W_j\le-\widehat t_j(q)\}.
 \label{eq:active-negative-set}
\end{equation}
Thus $\mathcal A_q$ contains the negative nulls whose magnitudes reach the cutoff obtained after forcing their own signs to be positive.  Define
\begin{align}
 I_q
 &=\E\left[
   \ind\{\mathcal A_q\ne\varnothing\}
   \frac{1}{|\mathcal A_q|}
   \sum_{j\in\mathcal A_q}\Omega_j
 \right],
 \label{eq:integrated-active-factor}\\
 M_{\Hnull}^-
 &=\E\left[
   1\vee\max_{\substack{j\in\Hnull\\W_j<0}}\Omega_j
 \right],
 \label{eq:negative-side-factor}
\end{align}
where the inner maximum in \eqref{eq:negative-side-factor} is zero when its index set is empty.
The factor $I_q$ follows the random active set and therefore depends on the testing level $q$.  The larger factor $M_{\Hnull}^-$ does not depend on $q$, which makes it more convenient for calibration.
At this stage, both factors are still theoretical quantities because they involve the unknown null set.  The Gaussian subsection below replaces $M_{\Hnull}^-$ by a null-set-free upper bound.

\begin{proposition}[FDR bounds under approximate null-sign symmetry]
\label{prop:approx-sign}
For the knockoff+ threshold \eqref{eq:threshold} run at nominal level $q$,
\begin{align}
 \FDR(q)
 &\le q+d_{\rm sign}
 \le q+\sqrt{K_{\rm sign}/2},
 \label{eq:tv-bound}\\
 \FDR(q)
 &\le e^\varepsilon q+
       \Prb(\Lambda_{\max}>\varepsilon),
 \qquad \varepsilon\ge0.
 \label{eq:odds-bound}
\end{align}
Under the additional positivity assumption above, the integrated bounds are
\begin{equation}
 \FDR(q)\le qI_q\le qM_{\Hnull}^-.
 \label{eq:integrated-odds-bound}
\end{equation}
\end{proposition}

The bounds are complementary.  Total variation is useful when the full conditional sign law is close to fair coins, the integrated form when each sign is only weakly predictable from the others, and the exceptional-event form when very large odds occur only rarely.  None is uniformly smallest.  None requires independent null signs; the integrated form instead needs the positivity condition stated above.

\begin{proof}[Proof of the total-variation and exceptional-event bounds]
For \eqref{eq:tv-bound}, construct a reference law $\Prb^\circ$ that keeps $\mathcal G$ unchanged and, conditionally on $\mathcal G$, replaces the null signs by independent fair coins.  Assumption~\ref{ass:signflip} holds under $\Prb^\circ$, so Proposition~\ref{prop:known} gives $\E^\circ(\FDP)\le q$.  Because the two laws have the same marginal law for $\mathcal G$, averaging their conditional total-variation distances gives
\[
 d_{\rm TV}(\Prb,\Prb^\circ)
 =\E\{d_{\rm TV}(\nu_{\mathcal G},\nu^\circ_{\mathcal G})\}
 =d_{\rm sign}.
\]
Since $0\le\FDP\le1$, its expectations under the two laws differ by at most this total-variation distance.  Pinsker's inequality and Jensen's inequality give the second inequality in \eqref{eq:tv-bound}.

For \eqref{eq:odds-bound}, put $B_j=\{\Lambda_j\le\varepsilon\}$.  This event is measurable with respect to $(|W_j|,W_{-j})$.  For every null $j$,
\begin{align*}
 &\Prb(W_j>0,B_j\mid |W_j|,W_{-j})\\
 &\quad=\ind\{\Lambda_j\le\varepsilon\}
       \Prb(W_j>0\mid |W_j|,W_{-j})
 \le e^\varepsilon
       \Prb(W_j<0\mid |W_j|,W_{-j}).
\end{align*}
Lemma~\ref{lem:leave-one-out-odds} therefore bounds the expected FDP contribution from nulls satisfying $B_j$ by $e^\varepsilon q$.  All remaining nulls contribute at most the indicator of $\{\Lambda_{\max}>\varepsilon\}$.  Taking expectations proves \eqref{eq:odds-bound} without any conditional-independence assumption.
\end{proof}

\begin{deferrednewproof}{C}{Proof of the integrated bound in Proposition~\ref{prop:approx-sign}}{app:proof-integrated-odds}
Write $T=\widehat t(W;q)$.  When $T<\infty$, the defining inequality for the threshold gives
\[
 \FDP
 \le q\,
 \frac{\sum_{j\in\Hnull}\ind\{W_j\ge T\}}
      {1+\sum_{j\in\Hnull}\ind\{W_j\le-T\}},
\]
because the denominator using only null negative scores is no larger than the denominator using all negative scores.
With indicators involving $T=\infty$ interpreted as zero, the same bound holds when $T=\infty$, because then both sides are zero.

Since $\widehat t_j(q)$ is measurable with respect to $(|W_j|,W_{-j})$,
\begin{align*}
 &\E\left[
 \frac{\sum_{j\in\Hnull}\ind\{W_j\ge T\}}
      {1+\sum_{j\in\Hnull}\ind\{W_j\le-T\}}
 \right]\\
 &=\sum_{j\in\Hnull}
 \E\left[
 \frac{\ind\{W_j>0,\ |W_j|\ge \widehat t_j(q)\}}
      {1+\sum_{k\in\Hnull,\,k\ne j}\ind\{W_k\le-\widehat t_j(q)\}}
 \right]\\
 &=\sum_{j\in\Hnull}
 \E\left[
 \frac{\Omega_j\ind\{W_j<0,\ |W_j|\ge \widehat t_j(q)\}}
      {1+\sum_{k\in\Hnull,\,k\ne j}\ind\{W_k\le-\widehat t_j(q)\}}
 \right].
\end{align*}
The first equality uses that, when $W_j>0$, replacing $W_j$ by $|W_j|$ does not change the score vector, whereas both numerator indicators vanish when $W_j<0$.  The second equality is the defining conditional-odds identity.

It remains to simplify the last sum.  We need two deterministic threshold facts.  First, all active coordinates have the same counterfactual threshold.  To see this, let $j,k\in\mathcal A_q$ and suppose, without loss of generality, that $\widehat t_j(q)\le\widehat t_k(q)$.  At the candidate threshold $t=\widehat t_j(q)$, the two vectors $W^{(j)}$ and $W^{(k)}$ have the same positive and negative counts: both $|W_j|$ and $|W_k|$ are at least $t$, and the pair contributes one positive and one negative score under either vector.  Since $t$ passes for $W^{(j)}$, it also passes for $W^{(k)}$, giving $\widehat t_k(q)\le\widehat t_j(q)$.  Thus equality holds.

Second, at this common threshold the selected negative nulls are exactly $\mathcal A_q$.  More precisely, for any $j\in\mathcal A_q$,
\[
 \{k\in\Hnull:W_k\le-\widehat t_j(q)\}=\mathcal A_q.
\]
Inclusion from right to left follows from equality of the active thresholds.  Conversely, if $W_k\le-\widehat t_j(q)$ and $\widehat t_k(q)>\widehat t_j(q)$, the same count comparison at $t=\widehat t_j(q)$ contradicts the minimality of $\widehat t_k(q)$; if $\widehat t_k(q)\le\widehat t_j(q)$, then $W_k\le-\widehat t_k(q)$ directly.

Therefore the last displayed sum equals zero when $\mathcal A_q$ is empty and otherwise equals
\[
 \frac{1}{|\mathcal A_q|}\sum_{j\in\mathcal A_q}\Omega_j.
\]
This proves the first inequality in \eqref{eq:integrated-odds-bound}.  Bounding the active-set average by the negative-side maximum proves the second.
\end{deferrednewproof}

\begin{newrevision}
The next result shows that both integrated factors are sharp even for exchangeable scores with standard normal margins and a positive joint density.

\begin{proposition}[Sharpness of the integrated factors]
\label{prop:integrated-sharp}
Fix $a\in[1/4,1/2)$.  For every $m\ge2$, put
\[
 q_m=\frac1m,
 \qquad
 c_m=\frac{1-2a}{2^m-2}.
\]
There is an exchangeable all-null score vector with standard normal margins, a joint density that can be chosen positive everywhere, and both conditional sign probabilities in \eqref{eq:conditional-log-odds} positive almost surely, such that
\begin{equation}
 \FDR(q_m)=q_mI_{q_m}=a,
 \qquad
 M_{\Hnull}^-=1+m(a-c_m).
 \label{eq:sharp-integrated-factors}
\end{equation}
Consequently, for every fixed $a\in[1/4,1/2)$,
\[
 \frac{\FDR(q_m)}{q_mM_{\Hnull}^-}
 =\frac{a}{a-c_m+1/m}\longrightarrow1
 \qquad(m\to\infty).
\]
Thus $\FDR(q_m)\le q_mI_{q_m}$ holds with equality, while
$\FDR(q_m)/(q_mM_{\Hnull}^-)\to1$ along $q_m=1/m\to0$.
\end{proposition}

The construction puts probability $a$ on each of the all-positive and all-negative sign vectors and spreads the remaining probability equally over all other sign vectors.  Independent magnitudes distributed as $|N(0,1)|$ then give standard normal margins.  After the scores are ordered by decreasing magnitude, only the complete prefix of all $m$ scores can pass at $q_m=1/m$, which makes the identities in \eqref{eq:sharp-integrated-factors} explicit.
\end{newrevision}

\begin{deferrednewproof}{C}{Proof of Proposition~\ref{prop:integrated-sharp}}{app:proof-integrated-sharp}
Let $U_1,\ldots,U_m$ be independent variables with distribution $|N(0,1)|$, independent of a sign vector $S$.  Give the all-positive and all-negative sign configurations probability $a$ each, and give every other sign configuration probability $c_m$.  Set $W_i=S_iU_i$.  This law is exchangeable and is unchanged by a global sign flip.  Away from the zero-coordinate set, its density is
\[
 2^m\Prb\{S=\operatorname{sign}(w)\}
 \prod_{i=1}^m\phi(|w_i|).
\]
It is positive there.  The zero-coordinate set has $m$-dimensional volume zero, so assigning positive density values there does not change the distribution.  The identity $2a+(2^m-2)c_m=1$ normalizes the sign law, and $a,c_m>0$.  Thus every sign vector has positive probability, and both conditional sign probabilities in \eqref{eq:conditional-log-odds} are positive almost surely.  Global sign symmetry makes each marginal sign fair; together with the independent $|N(0,1)|$ magnitude, this makes every $W_i$ standard normal.

At $q_m=1/m$, no prefix of length $k<m$ can pass, because even an all-positive prefix has ratio $1/k>1/m$.  At the full prefix, the ratio passes exactly when every sign is positive, so $\FDR(q_m)=a$.  If a configuration has $r$ negative signs, forcing one negative sign to be positive makes the counterfactual threshold pass exactly when $r=1$.  In that case its cutoff is $\min_iU_i$, so the unique negative coordinate is active.  Hence $\mathcal A_{q_m}$ is nonempty exactly when the sign vector has one negative coordinate, and then it contains that coordinate alone.  Its conditional positive-to-negative sign odds are $a/c_m$.  There are $m$ one-negative configurations, each with probability $c_m$, and therefore
\[
 I_{q_m}=mc_m\frac{a}{c_m}=ma,
 \qquad q_mI_{q_m}=a.
\]

Since $a\ge1/4$ and $m\ge2$, we have $c_m\le a$.  For $M_{\Hnull}^-$, the integrand equals $a/c_m$ on the $m$ one-negative configurations.  It equals one on the all-positive and all-negative configurations and on every configuration with between two and $m-1$ negative signs.  Hence
\[
 M_{\Hnull}^-
 =mc_m\frac{a}{c_m}+(1-mc_m)
 =1+m(a-c_m),
\]
which proves \eqref{eq:sharp-integrated-factors}.
\end{deferrednewproof}

\begin{newrevision}
At the conventional level $q=0.1$, take $m=10$ and let $a\uparrow1/2$.  Then
\[
 \frac{\FDR}{qM_{\Hnull}^-}\longrightarrow\frac56,
 \qquad
 \frac{qM_{\Hnull}^-}{\FDR}\longrightarrow\frac65.
\]
Thus, in this limiting family, $qM_{\Hnull}^-$ is only $20\%$ above the exact FDR even at $q=0.1$.
\end{newrevision}

Panels~\ref{fig:adjustment-numerics}(c)--(d) reuse the sign-mixture idea from the proof of Proposition~\ref{prop:integrated-sharp}, with smaller aligned-state probabilities and genuine signals.  They illustrate why the two calibrations are complementary: exceptional-event calibration can set aside a rare high-odds region, whereas integrated calibration averages those large odds over the possible negative controls.

The exceptional-event bound becomes a level adjustment if we allow the odds bound $\Gamma_\delta$ to be violated on an event of probability at most $\delta$.  For $0\le\delta<1$, define
\begin{equation}
 \Gamma_\delta
 =\inf\left\{\Gamma\in[1,\infty):
       \Prb(\Lambda_{\max}>\log\Gamma)\le\delta\right\},
 \label{eq:gamma-delta}
\end{equation}
with $\inf\varnothing=\infty$.  Whenever $\Gamma_\delta<\infty$, Proposition~\ref{prop:approx-sign} gives
\begin{equation}
 \FDR(q)\le \Gamma_\delta q+\delta.
 \label{eq:Cq-delta}
\end{equation}
Thus $\Gamma_\delta$ is the smallest multiplicative odds bound that may be violated with probability at most $\delta$.
For a target $\alpha>\delta$, the exceptional-event and integrated bounds give the two ideal input levels
\begin{equation}
 q_{\rm tail}=\frac{\alpha-\delta}{\Gamma_\delta},
 \qquad
 q_{\rm int}=\frac{\alpha}{M_{\Hnull}^-},
 \label{eq:ideal-adjusted-levels}
\end{equation}
with a ratio interpreted as zero when its denominator is infinite.  If both factors are known, one may run at
\begin{equation}
 q_{\rm run}=q_{\rm tail}\vee q_{\rm int},
 \label{eq:adjusted-q}
\end{equation}
because each branch individually justifies its proposed level.  Exact conditional sign flips correspond to $\Gamma_0=M_{\Hnull}^-=1$.

In practice, the two ideal factors may be unknown.  Independent calibration data can instead provide conservative upper bounds.  We state the two guarantees separately because they bound different objects: a high quantile of the worst conditional odds, or the mean negative-side odds.  In both results, $\eta$ bounds the probability that the calibration bound fails.  Because $\FDP\le1$, reserving $\eta$ of the target accounts for that event.

\begin{corollary}[Independent calibration of the integrated factor]
\label{cor:estimated-integrated-factor}
Under the additional positivity assumption in Proposition~\ref{prop:approx-sign}, fix $\alpha\in(0,1)$ and $\eta\in[0,\alpha)$.  Let $\mathcal D$ be independent of the testing scores, and let $\widehat M\in[1,\infty]$ be $\mathcal D$-measurable with
\[
 \Prb(M_{\Hnull}^-\le\widehat M)\ge1-\eta.
\]
Running at $q_{\rm run}=(\alpha-\eta)/\widehat M$, with $q_{\rm run}=0$ when $\widehat M=\infty$, gives $\FDR\le\alpha$.
\end{corollary}

\begin{proof}
On the calibration event, independence and \eqref{eq:integrated-odds-bound} give conditional FDR at most $M_{\Hnull}^-q_{\rm run}\le\alpha-\eta$.  On its complement, use $\FDP\le1$ and average over $\mathcal D$.
\end{proof}

\begin{corollary}[Independent calibration of the exceptional-event factor]
\label{cor:estimated-gamma}
Fix $\alpha\in(0,1)$ and $\delta,\eta\in[0,1)$ with $\delta+\eta<\alpha$.  Let $\mathcal D$ be independent of the testing scores.  Suppose $\widehat\Gamma\in[1,\infty]$ is $\mathcal D$-measurable and, on an event $A\in\sigma(\mathcal D)$ with $\Prb(A)\ge1-\eta$,
\[
 \Prb(\Lambda_{\max}>\log\widehat\Gamma\mid\mathcal D)
 \le\delta
 \quad\text{whenever }\widehat\Gamma<\infty.
\]
Running at $q_{\rm run}=(\alpha-\delta-\eta)/\widehat\Gamma$, with $q_{\rm run}=0$ when $\widehat\Gamma=\infty$, gives $\FDR\le\alpha$.
\end{corollary}

In words, with probability at least $1-\eta$ over the calibration data, $\widehat\Gamma$ leaves at most $\delta$ probability under a fresh independent testing draw in the exceptional region.

\begin{proof}
On $A$, \eqref{eq:odds-bound} gives conditional FDR at most $\alpha-\eta$; if $\widehat\Gamma=\infty$, the procedure returns no rejections.  On $A^c$, use $\FDP\le1$ and average over $\mathcal D$.
\end{proof}

If both upper bounds hold on one joint event of probability at least $1-\eta$, the two estimated levels can again be combined:
\begin{equation}
 q_{\rm run}
 =\frac{\alpha-\delta-\eta}{\widehat\Gamma}
 \vee\frac{\alpha-\eta}{\widehat M}.
 \label{eq:estimated-combined-level}
\end{equation}
If this maximum is zero, the procedure returns no rejections.  Otherwise the larger finite branch gives conditional FDR at most $\alpha-\eta$ on the joint event, and $\FDP\le1$ handles its complement.  If the two upper bounds are proved on separate events, their failure probabilities must be added when defining $\eta$.

A concrete non-Gaussian implementation is possible when the same score construction can be applied to several separate datasets generated under a known all-null condition, such as negative-control experiments or null-condition batches.  These datasets must be mutually independent and independent of the testing data; repeatedly recomputing scores from the testing data does not qualify.  Coordinates within each calibration dataset may still be dependent.  Appendix~\ref{app:negative-control-calibration} gives the construction and assumptions in detail.  The next subsection instead derives a bound from a specified Gaussian covariance model.

\subsection{Gaussian calibration from the covariance matrix}
\label{subsec:gaussian-calibration}

Under a specified Gaussian model, the covariance matrix determines every conditional normal distribution and therefore the relevant sign odds.  This turns the abstract conditional-odds bounds above into computable quantities.  Let $m\ge2$ and
\[
 W\sim N_m(\mu,\Sigma),
 \qquad \Sigma_{jj}=1,
 \qquad \Sigma\text{ positive definite},
\]
and take $\Hnull=\{j:\mu_j=0\}$.  A subscript $-j$ means that coordinate $j$, or the corresponding row or column, is removed.  For $j=1,\ldots,m$, write
\[
 b_j=\Sigma_{-j,-j}^{-1}\Sigma_{-j,j},
 \qquad
 \tau_j^2=1-\Sigma_{j,-j}b_j,
 \qquad
 r_j^2=\Sigma_{j,-j}\Sigma_{-j,-j}^{-1}\Sigma_{-j,j}
       =1-\tau_j^2,
\]
and let $r_j$ denote the nonnegative square root of $r_j^2$, with $r_\star=\max_jr_j<1$.  Here $b_j$ is the Gaussian regression coefficient, $\tau_j^2$ is the conditional variance of coordinate $j$, and $b_j^{\mathsf T}(W_{-j}-\mu_{-j})$ is the conditional mean of a null coordinate $W_j$.  Because every marginal variance is one, $r_j^2$ is the coefficient of determination from regressing coordinate $j$ on the remaining coordinates.  Thus $r_\star^2=\max_jr_j^2$ is the largest fraction of any coordinate's variance explained collectively by all the others; it is not the largest pairwise squared correlation.

The calibration uses a generic centered Gaussian draw $X\sim N_m(0,\Sigma)$.  Define the coordinatewise log-odds statistic by
\begin{equation}
 \ell_j(X)=\frac{2|X_j|b_j^{\mathsf T}X_{-j}}{\tau_j^2},
 \qquad
 L_\Sigma(X)=\max_{1\le j\le m}\{\ell_j(X)\}_+,
 \label{eq:gaussian-odds-max}
\end{equation}
where $\{x\}_+=\max\{x,0\}$.  Define the integrated negative-side statistic and its mean by
\begin{equation}
 \Omega_\Sigma^-(X)
 =\exp\left[
   \max\left\{0,\max_{\substack{1\le j\le m\\X_j<0}}\ell_j(X)\right\}
 \right],
 \qquad
 M_\Sigma^-=\E\{\Omega_\Sigma^-(X)\},
 \label{eq:gaussian-integrated-statistic}
\end{equation}
where the inner maximum is $-\infty$ if every coordinate is positive.  We maximize over all coordinates because the null set $\Hnull$ is unknown.  The statistic $L_\Sigma$ is the largest positive conditional log-odds across coordinates.  The statistic $\Omega_\Sigma^-$ exponentiates the largest such value among coordinates that are negative in the current draw; its mean gives the integrated factor.

\begin{corollary}[Quantile and integrated Gaussian calibration]
\label{cor:gaussian-calibration}
\label{cor:integrated-gaussian}
Let $0\le\delta<1$ and let $\kappa_{\Sigma,\delta}\ge0$ satisfy
\[
 \Prb\{L_\Sigma(X)>\kappa_{\Sigma,\delta}\}\le\delta.
\]
Then, for every mean vector $\mu$,
\begin{equation}
 \FDR(q)\le
 \min\left\{
 e^{\kappa_{\Sigma,\delta}}q+\delta,
 M_\Sigma^-q
 \right\}.
 \label{eq:gaussian-quantile-bound}
\end{equation}
Moreover, $1\le M_\Sigma^-\le1+m/2$.  The distributions of $L_\Sigma(X)$ and $\Omega_\Sigma^-(X)$ depend only on $\Sigma$; the unknown null set and nonnull means do not enter.
\end{corollary}

Thus, for exact calibration factors and a target $\alpha>\delta$, one may run at
\begin{equation}
 q_{\rm run}
 =(\alpha-\delta)e^{-\kappa_{\Sigma,\delta}}
 \vee\frac{\alpha}{M_\Sigma^-}.
 \label{eq:gaussian-combined-level}
\end{equation}
For every positive-definite $\Sigma$, $M_\Sigma^-$ is finite, and $\kappa_{\Sigma,\delta}$ can be chosen finite whenever $\delta>0$; which branch of \eqref{eq:gaussian-combined-level} is larger depends on $\Sigma$.

The relevant quantile of $L_\Sigma(X)$ and the mean $M_\Sigma^-$ can be approximated using independent synthetic draws from $N_m(0,\Sigma)$; these are computer-generated draws, not reused testing data.  With finitely many draws, however, a raw empirical quantile or sample mean need not bound the population quantity in the conservative direction.  Appendix~\ref{app:finite-sample-gaussian-computation} therefore constructs finite-sample-valid one-sided upper confidence bounds: an order-statistic bound for the quantile of $L_\Sigma(X)$ and a moment-truncation bound for $M_\Sigma^-$.

The moment-truncation construction allows any prespecified moment order $s>1$ that satisfies its stated condition.  For the special choice $s=4$, that condition is $r_\star^2<1/49$.  Because $r_\star^2$ is the largest regression $R^2$ defined above, this says that less than $1/49\approx2.04\%$ of any coordinate's variance may be explained by the remaining coordinates.  This restriction belongs only to the fourth-moment implementation, not to Corollary~\ref{cor:gaussian-calibration} or its integrated Gaussian bound.  Other choices of $s$ allow different dependence ranges, although the resulting upper confidence bound can become conservative as coordinatewise predictability increases.

If $\Sigma$ is estimated from independent data, one must use calibration upper bounds that hold uniformly over a confidence set for $\Sigma$.  The confidence-set failure probability must then be added to $\eta$ in the relevant independent-calibration result and, when the two routes are combined, in the joint budget of \eqref{eq:estimated-combined-level}.

\begin{deferredproof}{C}{Proof of Corollary~\ref{cor:gaussian-calibration}}{app:proof-gaussian-calibration}
\label{app:proof-integrated-gaussian}
For a null $j$, Gaussian conditioning gives
\[
 W_j\mid W_{-j}=w_{-j}
 \sim N\{b_j^{\mathsf T}(w_{-j}-\mu_{-j}),\tau_j^2\}.
\]
Writing $X=W-\mu$, the density ratio at $|W_j|=|X_j|$ therefore gives
\[
\Lambda_j
 =\frac{2|X_j|b_j^{\mathsf T}X_{-j}}{\tau_j^2}
 \le L_\Sigma(X).
\]
Consequently,
\[
 \Prb(\Lambda_{\max}>\kappa_{\Sigma,\delta})
 \le \Prb\{L_\Sigma(X)>\kappa_{\Sigma,\delta}\}
 \le\delta,
\]
and the first assertion follows from \eqref{eq:odds-bound}.

For the integrated assertion, a null coordinate satisfies $W_j=X_j$, and Gaussian conditioning gives $\Omega_j=e^{\ell_j(X)}$.  Therefore the negative-side maximum over the unknown null set is bounded by $\Omega_\Sigma^-(X)$.  Proposition~\ref{prop:approx-sign} gives $\FDR(q)\le M_\Sigma^-q$.

To bound the factor, note that
\[
 \Omega_\Sigma^-(X)
 \le1+\sum_{j=1}^m e^{\ell_j(X)}\ind\{X_j<0\}.
\]
Conditional on $(|X_j|,X_{-j})$, the definition of the sign odds gives
\[
 e^{\ell_j(X)}\Prb(X_j<0\mid |X_j|,X_{-j})
 =\Prb(X_j>0\mid |X_j|,X_{-j}).
\]
Hence
\[
 \E\{e^{\ell_j(X)}\ind\{X_j<0\}\}
 =\Prb(X_j>0)=\frac12.
\]
Summing over $j$ proves $1\le M_\Sigma^-\le1+m/2$.
\end{deferredproof}

For the equicorrelation model \eqref{eq:equicorr},
\[
r_\star^2
 =\frac{(m-1)\rho^2}{1+(m-2)\rho}.
\]
For fixed $\rho>0$, this quantity approaches $\rho$ as $m$ grows.  Thus even small pairwise correlations can collectively make one coordinate appreciably predictable from all the others, so the uniform finite-sample bounds can be conservative.

\begin{remark}[Exact Gaussian sign-flip repair by adding noise]
\label{rem:gaussian-whitening}
A specified Gaussian model can also be used to change the scores rather than calibrate their conditional sign odds.  Let $D=\operatorname{diag}(d_1,\ldots,d_m)$ be fixed before testing, with $d_j>0$ for every $j$ and $D-\Sigma$ positive semidefinite.  Independently draw
\[
 \omega\sim N_m(0,D-\Sigma),
 \qquad
 \widetilde W=W+\omega,
 \qquad
 Z=D^{-1/2}\widetilde W.
\]
Then $Z\sim N_m(D^{-1/2}\mu,I_m)$.  Its coordinates are independent, and every null $Z_j$ is standard normal regardless of the other means.  Consequently, the null signs satisfy Assumption~\ref{ass:signflip} exactly, and knockoff+ applied to $Z$ at $q=\alpha$ controls the FDR at level $\alpha$.  Such a $D$ always exists; for example, one may take $D=\lambda_{\max}(\Sigma)I_m$.

The remaining information can still guide the ordering and anticipated directions.  Define
\[
 \xi=\Sigma^{-1}W-D^{-1}\widetilde W.
\]
Since $\operatorname{Cov}(\xi,\widetilde W)=0$, joint Gaussianity shows that $\xi$ and $\widetilde W$ are independent.  Thus the ordering and directions may depend on $(\xi,|\widetilde W|)$, but not on the signs of $\widetilde W$ used for testing.  This is the whitening representation of fixed-$X$ knockoffs studied by \citet{li-fithian-2021}; it is not a new construction here.

Exact sign flips can nevertheless be costly.  For the equicorrelation covariance used in Figure~\ref{fig:adjustment-numerics}, consider choices $D=dI_m$ that treat every coordinate alike.  The smallest feasible $d$ is
\[
 d=1+(m-1)\rho.
\]
At the endpoint $m=5000$ and $\rho=0.05$, this gives $d=250.95$ and reduces an alternative mean of $4.5$ to $4.5/\sqrt d=0.284$ after standardization.  Whitening therefore keeps the input level at $\alpha$ but can weaken the testing signs; the odds calibration reduces the input level while retaining the original scores.
\end{remark}

The covariance matrix can also be known from the study design.  For example, let $Y\sim N_n(\theta,\sigma^2I_n)$ and let $a_1,\ldots,a_m$ be prespecified nonzero contrast vectors whose normalized Gram matrix is positive definite.  Define
\[
 W_j=\frac{a_j^{\mathsf T}Y}{\|a_j\|},
 \qquad
 \Sigma_{jk}=\frac{a_j^{\mathsf T}a_k}{\|a_j\|\,\|a_k\|},
 \qquad
 \Hnull=\{j:a_j^{\mathsf T}\theta=0\}.
\]
Then $W/\sigma$ has unit marginal variances and correlation matrix $\Sigma$, which is fixed by the design before $Y$ is observed.  The Gaussian calibration applies to $W/\sigma$; because multiplying all scores by the same positive constant does not change the rejection set, the unknown scale $\sigma$ need not be known.  Thus the calibration uses neither null identities nor signal sizes.

\subsection{Why dependence information is indispensable}

The Gaussian correction works because the covariance matrix is embedded in a fully specified joint model.  The same matrix, used only as a distribution-free summary, need not say anything useful about the FDR.

\begin{proposition}[The covariance matrix alone does not calibrate the FDR]
\label{prop:no-low-order-calibration}
Fix $q\in(0,1)$.  For every $\eta>0$, there are an integer $m$ and two all-null score vectors with the same common continuous symmetric marginal distribution and the same covariance matrix: one has independent coordinates and FDR at most $q$, while the other is exchangeable and pairwise uncorrelated and has FDR greater than $1-\eta$.
\end{proposition}

\begin{proof}
Take the dependent vector from Theorem~\ref{thm:nearone} and compare it with independent coordinates drawn from its common marginal distribution.  The two vectors have the same marginals and diagonal covariance matrix.  Under independence, symmetry makes the null signs independent fair coins conditional on their magnitudes, so Proposition~\ref{prop:known} applies.  The dependent vector has FDR greater than $1-\eta$ by Theorem~\ref{thm:nearone}.
\end{proof}

Other low-dimensional summaries have similar limitations.  Proposition~\ref{prop:nongaussian} rules out a uniform bound below one half based only on PRDS, while Theorem~\ref{thm:fixedrho} shows the exact limiting value one half at every fixed positive equicorrelation $\rho$.

There is also a simple finite-dimensional obstruction.  Reducing the nominal level is nontrivial only when $q_{\rm run}\ge1/m$.  At any such level, the all-positive PRDS construction can have FDR arbitrarily close to $1/2$; if $q_{\rm run}<1/m$, the rule cannot reject.

Finally, a deterministic almost-sure odds constant is generally unavailable even in a known Gaussian model.  Write $m_j(w_{-j})$ for the conditional mean.  If
\[
 W_j\mid W_{-j}=w_{-j}
 \sim N\{m_j(w_{-j}),\tau_j^2\},
\]
then, for $u>0$,
\[
 \frac{\Prb(W_j>0\mid |W_j|=u,W_{-j}=w_{-j})}
      {\Prb(W_j<0\mid |W_j|=u,W_{-j}=w_{-j})}
 =\exp\left\{\frac{2u m_j(w_{-j})}{\tau_j^2}\right\}.
\]
Whenever $W_j$ depends on $W_{-j}$, the conditional mean is a nonconstant linear function with unbounded range.  Under a full-support Gaussian law, the smallest almost-sure upper bound (the essential supremum) of the odds ratio is therefore infinite.  The quantile bound avoids this obstruction by excluding a controlled exceptional event.  The integrated bound avoids it by averaging a negative-side odds maximum.  Both require information about the joint law that is absent from marginal symmetry or covariance alone.
\end{revision}

\FloatBarrier
\section{Numerical illustrations}
\label{sec:numerical}

\begin{revision}
Figure~\ref{fig:numerics} illustrates the main failure mechanisms at finite sample sizes.
\end{revision}

\begin{figure}[t]
 \centering
 \begin{minipage}[t]{0.47\linewidth}
  \centering
  \includegraphics[width=\linewidth]{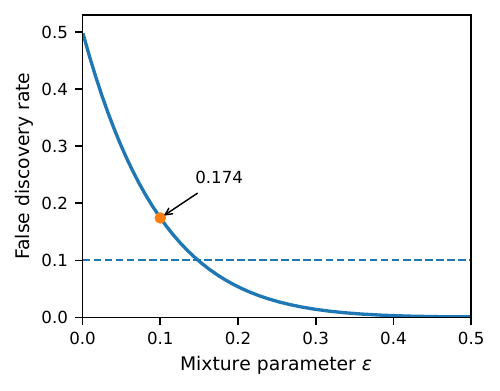}\\[-1mm]
  {\small (a) Non-Gaussian PRDS model}
 \end{minipage}
 \hfill
 \begin{minipage}[t]{0.47\linewidth}
  \centering
  \includegraphics[width=\linewidth]{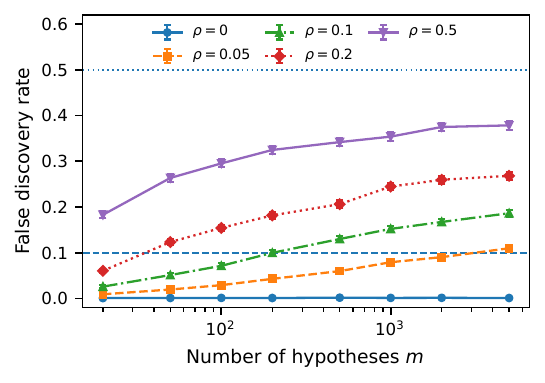}\\[-1mm]
  {\small (b) Fixed Gaussian equicorrelation}
 \end{minipage}

 \vspace{1.0ex}
 \begin{minipage}[t]{0.47\linewidth}
  \centering
  \includegraphics[width=\linewidth]{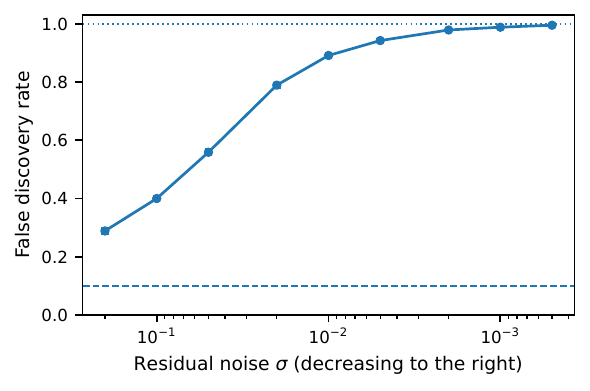}\\[-1mm]
 {\small (c) Exchangeable pairwise-uncorrelated scores}
 \end{minipage}
 \hfill
 \begin{minipage}[t]{0.47\linewidth}
  \centering
  \includegraphics[width=\linewidth]{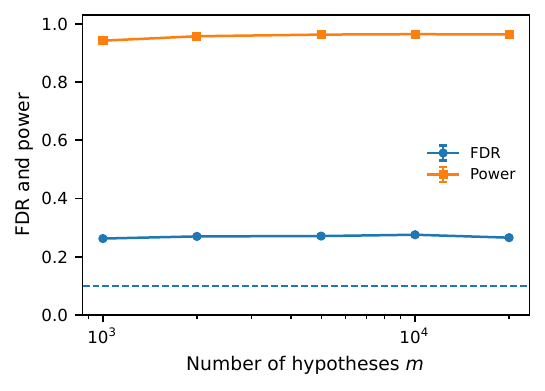}\\[-1mm]
  \rev{{\small (d) Opposing Gaussian loadings with alternatives}}
 \end{minipage}
 \caption{\rev{Finite-sample behavior at $q=0.1$.  Panels (a)--(c) are global-null; panel (d) has a $2\%$ alternative fraction.  Panel (a) is exact ($m=11$); panels (b)--(d) use $10{,}000$ Monte Carlo repetitions.  Error bars are pointwise 95\% Wilson intervals for binomial proportions in panels (b)--(c) \citep{wilson-1927} and mean $\pm1.96$ Monte Carlo standard errors in panel (d).  Dashed lines mark $q$; dotted lines mark the Theorem~\ref{thm:fixedrho} limit $1/2$ in panel (b) and the near-one limit from Theorem~\ref{thm:nearone} in panel (c), obtained by first increasing $m$ and then decreasing $\sigma$.}}
 \label{fig:numerics}
\end{figure}

\begin{revision}
Panel~(a) plots the exact PRDS calculation in \eqref{eq:prds-exact}.  Panel~(b) uses standard Gaussian scores with fixed equicorrelation.  At $\rho=0.1$, the empirical FDR rises from $0.0996$ at $m=200$ to $0.1520$ at $m=1000$; at $\rho=0.05$, it reaches $0.1096$ by $m=5000$.  For comparison, the $\rho=0$ curve is highly conservative.  The visible finite-$m$ gap from the limiting value one half is consistent with the slow one-sided rate in Remark~\ref{rem:rate}.

Panel~(c) uses \eqref{eq:exchangeable-factor} with $m=500$, $\pi=0.05$, $a=1$, and $b=1/19$, so distinct scores are uncorrelated.  As $\sigma$ decreases from $0.2$ to $0.0005$, the empirical FDR rises from $0.2887$ to $0.9949$.  The two intermediate values are $0.5587$ at $\sigma=0.05$ and $0.8910$ at $\sigma=0.01$.  At every displayed value, the coordinates have the same continuous symmetric marginal law and the joint density is positive.
\end{revision}

\begin{revision}
Panel~(d) is a moderate-dimensional illustration of the opposing-loading mechanism and its coexistence with high power.  It uses the two null blocks in Theorem~\ref{thm:gaussian-nearone} with $a=0.95$, $c=0.35$, second-block proportion $\pi=0.20$, alternative fraction $\gamma=0.02$, and alternative mean $\mu=4$.  The largest absolute null correlation is $a^2=0.9025$.  At $m=1000$, the empirical FDR and power are $0.2625$ and $0.9418$; at $m=20000$, they are $0.2656$ and $0.9631$.  It does not demonstrate near-total FDR or the theorem's arbitrary-small-correlation limit; Remark~\ref{rem:gaussian-nearone-dimension} explains why.
\end{revision}

\FloatBarrier
\begin{newrevision}
\subsection{Finite-sample calibration comparisons}
\label{subsec:calibration-comparisons}
Figure~\ref{fig:adjustment-numerics} compares the two adjustments in settings where correction is genuinely needed.

\begin{figure}[!htbp]
 \color{black}
 \centering
 \includegraphics[width=\linewidth]{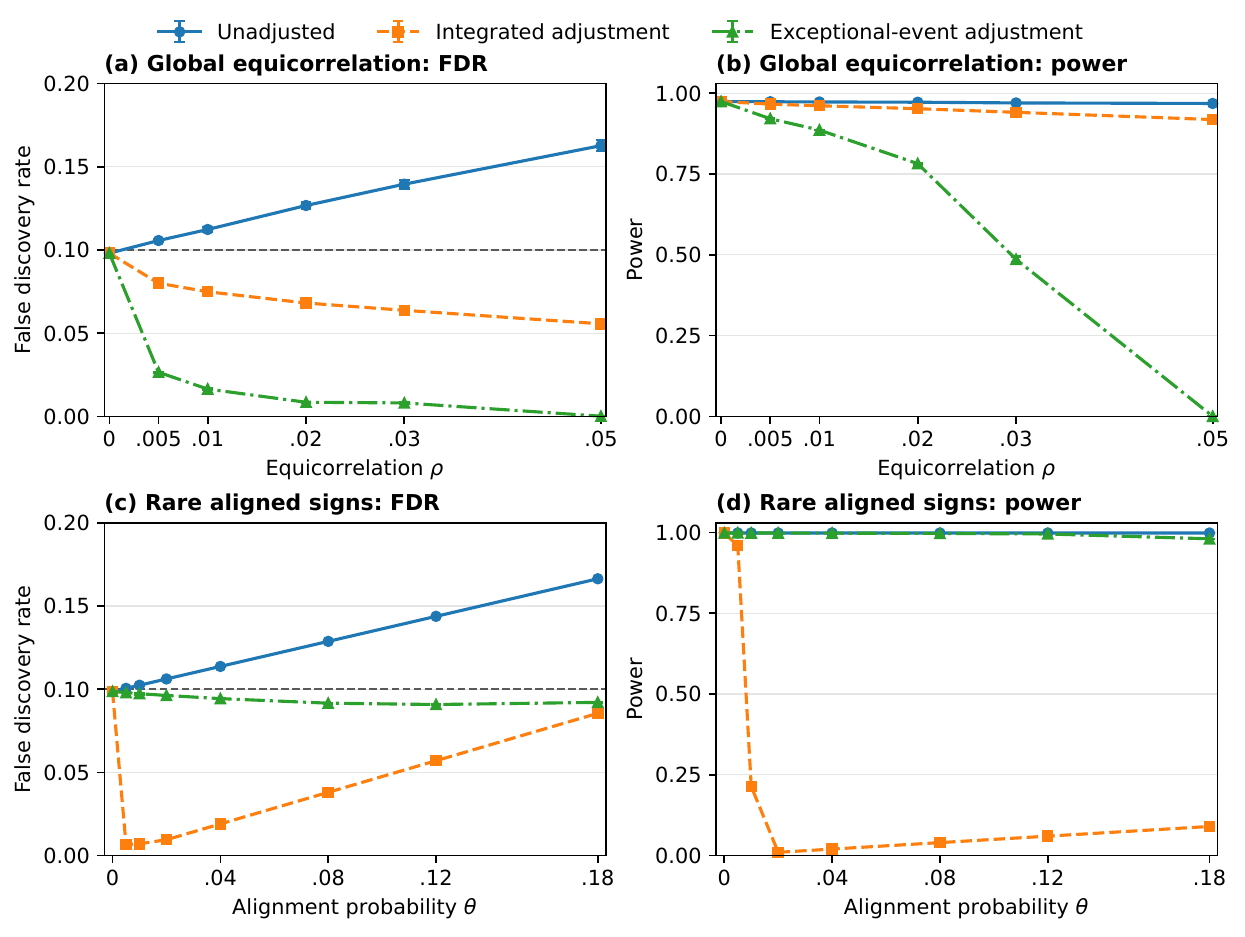}
 \caption{Calibration comparisons at target level $0.1$.  The top row uses global Gaussian equicorrelation; the bottom row uses the rare aligned-sign mixture.  Both models have $m=5000$ and 250 alternatives; Appendix~\ref{app:finite-sample-gaussian-computation} gives the complete model and computational details.  The columns show FDR and power.  Curves correspond to the unadjusted, integrated, and exceptional-event procedures.  Top-row values and error bars condition on the realized independent calibration sample; the finite-sample Gaussian guarantees average over both calibration and testing randomness.  The bottom row uses the exact population factors and stratified Monte Carlo.  Error bars show the mean $\pm1.96$ Monte Carlo standard errors, and dashed horizontal lines mark the target FDR.}
 \label{fig:adjustment-numerics}
\end{figure}

The two adjusted procedures are plotted separately to show what each bound contributes.  In practice, when both factors are known, or when their estimates are valid on one joint calibration event independent of the testing scores, one may run at $q_{\rm run}=q_{\rm int}\vee q_{\rm tail}$ as in \eqref{eq:adjusted-q} and \eqref{eq:estimated-combined-level}.  This choice uses only calibration information, not testing outcomes.  Because the rejection set is monotone in $q$, the jointly valid maximum gives the larger rejection set.  We do not add a fourth curve: for the separately calibrated levels in Figure~\ref{fig:adjustment-numerics}, the maximum coincides with the integrated curve at every positive $\rho$ and with the exceptional-event curve at every positive $\theta$; all three procedures coincide at independence.  Separate calibration failure probabilities must be included in the common budget, and the maximum must then be formed from the resulting recalibrated levels; Appendix~\ref{app:finite-sample-gaussian-computation} explains this point for the displayed Gaussian branches.

At $\rho=0.05$, the unadjusted Gaussian FDR is $0.1627$.  Integrated calibration has FDR $0.0556$ and power $0.9188$, whereas exceptional-event power is $0.0002$.  As an unplotted diagnostic, we replayed the same $10{,}000$ Gaussian testing draws and added independent noise for the exact repair in Remark~\ref{rem:gaussian-whitening}.  At $q=0.1$, the exploration-assisted score $W_j^{\rm wh}\propto\xi_j\widetilde W_j$ had power $0.0058$; even the infeasible benchmark $W_j^{\rm oracle}=\mu_j\widetilde W_j$, which knows the support and effect directions, had power $0.0081$.  This diagnostic uses the equal-diagonal choice $D=dI_m$ and does not compare all unequal-diagonal completions.  At $\theta=0.18$, the unadjusted aligned-sign FDR is $0.1664$.  Exceptional-event calibration has FDR $0.0921$ and power $0.9805$, whereas integrated power is $0.0900$.  Both adjustments keep the empirical FDR below $0.1$ throughout both sweeps.  Which bound is useful depends on whether sign predictability is diffuse or concentrated in a rare region.

\FloatBarrier
\end{newrevision}

\begin{revision}
\section{Implications and discussion}
\label{sec:discussion}

The central issue is whether the remaining data can predict a null score's sign after its magnitude and the genuine-signal scores are fixed.  If so, aligned signs can select a cutoff with too many false discoveries; marginal symmetry and small pairwise correlations do not rule this out.

The results do not say that every mirror method fails under dependence.  Validity comes from how the scores are constructed, not from the two-count formula alone.  The model-X knockoff filter obtains conditional sign flips by construction; other procedures may use sample splitting or a specified joint model.  Without comparable information, the adaptive cutoff is not justified.

When exact sign flips are unavailable, valid repairs require joint sign information: a global sign-law discrepancy, an exceptional-event odds bound, or a negative-side integrated odds bound.  Proposition~\ref{prop:integrated-sharp} shows that neither integrated factor is merely a proof artifact.  Without joint sign information, no count-only rule studied here can both make discoveries and control the FDR uniformly.
\end{revision}

\FloatBarrier
\section*{Acknowledgments}
The author used OpenAI GPT-5.6 (Sol) and Anthropic Claude (Opus 5.0) for language editing, proof exploration, and the numerical studies.  The author checked every definition, proof, and numerical result and is responsible for the paper.

\appendix

\begin{revision}
\section{First-passage proofs}
\label{app:first-passage}
\end{revision}
\printappendixproofs{A}

\begin{revision}
\section{Fixed-correlation and near-total FDR proofs}
\label{app:nearone}
\end{revision}
\printappendixproofs{B}

\begin{newrevision}
\section{Sign-odds proofs and finite-sample calibration}
\label{app:gaussian-calibration}
\end{newrevision}
\printappendixproofs{C}

\begin{newrevision}
\subsection{Finite-sample calibration and computation}
\label{app:finite-sample-gaussian-computation}

Recall that $L_\Sigma(X)$ is the Gaussian log-odds maximum, $\Omega_\Sigma^-(X)$ is the negative-side odds statistic, and $r_\star^2$ is the largest multiple-regression $R^2$.

\begin{corollary}[Finite-sample Gaussian Monte Carlo calibration]
\label{cor:finite-sample-gaussian}
Let $B\ge2$.  All synthetic draws below are independent of the testing scores.
\begin{enumerate}
 \item Fix $\delta,\eta\in(0,1)$.  Generate independent values $L^{(1)},\ldots,L^{(B)}$ from the law of $L_\Sigma(X)$ and write $L_{(1)}\le\cdots\le L_{(B)}$.  If there is a $k\in\{1,\ldots,B\}$ satisfying
 \begin{equation}
  \Prb\{\operatorname{Bin}(B,1-\delta)\ge k\}\le\eta,
  \label{eq:quantile-order-statistic-condition}
 \end{equation}
 let $k$ be the smallest such integer and set $\widehat\kappa=L_{(k)}$.  If no such integer exists, set $\widehat\kappa=\infty$.  Then
 \[
  \Prb\left[
    \Prb\{L_\Sigma(X)>\widehat\kappa\}\le\delta
  \right]\ge1-\eta.
 \]
 Here the outer probability is over the $B$ calibration draws, whereas the inner probability is over a fresh centered Gaussian draw $X$.  A finite qualifying order statistic exists exactly when $(1-\delta)^B\le\eta$.  For $\alpha>\delta+\eta$, the resulting finite-sample-valid testing level is
 \[
  q_{\rm tail}=(\alpha-\delta-\eta)e^{-\widehat\kappa},
 \]
 with $e^{-\infty}=0$.

 \item Fix $\eta\in(0,1)$.  Let $\mathcal S\subset(1,\infty)$ and $\mathcal C\subset[1,\infty)$ be finite prespecified sets of moment orders and truncation levels.  Define the admissible pair set
 \[
  \mathcal P_\Sigma
  =\left\{(s,h)\in\mathcal S\times\mathcal C:
    4s(s-1)r_\star^2<1-r_\star^2
   \right\},
  \qquad J=|\mathcal S|\,|\mathcal C|.
 \]
 For each $s\in\mathcal S$ satisfying $4s(s-1)r_\star^2<1-r_\star^2$, put
 \begin{equation}
  C_{s,\Sigma}
  =1+\frac12\sum_{j=1}^m
  \left(
    1-\frac{4s(s-1)r_j^2}{1-r_j^2}
  \right)^{-1/2}.
  \label{eq:gaussian-moment-bound}
 \end{equation}
 Generate independent values $Y^{(1)},\ldots,Y^{(B)}$ from the law of $\Omega_\Sigma^-(X)$.  For $h\in\mathcal C$, set
 \[
  Z_h^{(b)}=Y^{(b)}\wedge h,
  \qquad
  \overline Z_h=\frac1B\sum_{b=1}^B Z_h^{(b)},
  \qquad
  S_h^2=\frac1{B-1}\sum_{b=1}^B(Z_h^{(b)}-\overline Z_h)^2.
 \]
 If $\mathcal P_\Sigma\ne\varnothing$, define
 \begin{equation}
  \widehat M_{\Sigma,U}
  =1\vee\min_{(s,h)\in\mathcal P_\Sigma}
  \left\{
   \overline Z_h
   +\sqrt{\frac{2S_h^2\log(2J/\eta)}{B}}
   +\frac{7h\log(2J/\eta)}{3(B-1)}
   +\frac{C_{s,\Sigma}}{(s-1)h^{s-1}}
  \right\};
  \label{eq:integrated-monte-carlo-ucb}
 \end{equation}
 For each candidate pair, the four displayed terms are the truncated sample mean, an empirical-Bernstein variance margin, a bounded-range margin, and a deterministic bound on the omitted tail, respectively.
 If $\mathcal P_\Sigma=\varnothing$, set $\widehat M_{\Sigma,U}=\infty$.  Then
 \[
 \Prb(M_\Sigma^-\le\widehat M_{\Sigma,U})\ge1-\eta.
 \]
 Here the probability is over the $B$ synthetic calibration draws.
 For a target $\alpha>\eta$, the level
\begin{equation}
 q_{\rm int}=\frac{\alpha-\eta}{\widehat M_{\Sigma,U}}
 \label{eq:integrated-finite-sample-level}
 \end{equation}
 controls the FDR at $\alpha$, with $q_{\rm int}=0$ when $\widehat M_{\Sigma,U}=\infty$.
\end{enumerate}
The two estimated levels may be maximized only when both upper-bound events hold simultaneously, with their total failure probability charged as in \eqref{eq:estimated-combined-level}.
\end{corollary}

\begin{proof}
For the first assertion, let $F$ be the distribution function of $L_\Sigma(X)$, put $p=1-\delta$, and define $A=\{x:F(x)<p\}$.  Its probability under the law of $L_\Sigma(X)$ is at most $p$.  Because $F$ is nondecreasing, $A$ is an interval extending to the left.  Failure of the asserted tail inequality implies $F(\widehat\kappa)<p$.  If $\widehat\kappa=L_{(k)}$, then $L_{(k)}\in A$, so at least $k$ of the $B$ calibration draws lie in $A$.  The failure probability is therefore at most $\Prb\{\operatorname{Bin}(B,p)\ge k\}\le\eta$ by \eqref{eq:quantile-order-statistic-condition}.  If no $k$ exists, $\widehat\kappa=\infty$ and the tail inequality is automatic.  Since the smallest possible binomial tail over $k\in\{1,\ldots,B\}$ is its value at $k=B$, a finite $k$ exists exactly when $p^B=(1-\delta)^B\le\eta$.  On the upper-bound event, set $\widehat\Gamma=e^{\widehat\kappa}$; Corollary~\ref{cor:estimated-gamma} then gives the stated testing level.

For the integrated assertion, let $V_j=b_j^{\mathsf T}X_{-j}=\E(X_j\mid X_{-j})$.  Then $V_j\sim N(0,r_j^2)$ and $X_j=V_j+\varepsilon_j$, where the residual $\varepsilon_j\sim N(0,1-r_j^2)$ is independent of $V_j$.  For every admissible $s>1$, completing the square conditional on $V_j=v$ gives
\begin{align*}
 &\E\left[e^{s\ell_j(X)}\ind\{X_j<0\}\mid V_j=v\right]\\
 &\qquad=
 \exp\left\{\frac{2s(s-1)v^2}{1-r_j^2}\right\}
 \Phi\left(\frac{(2s-1)v}{\sqrt{1-r_j^2}}\right).
\end{align*}
The first factor is even in $v$, while $\Phi(z)+\Phi(-z)=1$.  Symmetry of $V_j$ and the Gaussian square-moment formula therefore give
\begin{equation}
 \E\left[e^{s\ell_j(X)}\ind\{X_j<0\}\right]
 =\frac12
 \left(
  1-\frac{4s(s-1)r_j^2}{1-r_j^2}
 \right)^{-1/2}.
 \label{eq:gaussian-odds-moment}
\end{equation}
Since
\[
 \{\Omega_\Sigma^-(X)\}^s
 \le1+\sum_{j=1}^m e^{s\ell_j(X)}\ind\{X_j<0\},
\]
equations \eqref{eq:gaussian-moment-bound} and \eqref{eq:gaussian-odds-moment} imply
$\E\{(\Omega_\Sigma^-)^s\}\le C_{s,\Sigma}$.

For a fixed pair $(s,h)$, the empirical Bernstein inequality of \citet{maurer-pontil-2009}, applied to $Z_h^{(b)}\in[0,h]$, gives with failure probability at most $\eta/J$,
\[
 \E(Z_h)
 \le\overline Z_h
 +\sqrt{\frac{2S_h^2\log(2J/\eta)}{B}}
 +\frac{7h\log(2J/\eta)}{3(B-1)}.
\]
Moreover,
\[
 M_\Sigma^- -\E(Z_h)
 =\int_h^\infty\Prb(\Omega_\Sigma^->t)\,dt
 \le\frac{\E\{(\Omega_\Sigma^-)^s\}}
          {(s-1)h^{s-1}}
 \le\frac{C_{s,\Sigma}}{(s-1)h^{s-1}}.
\]
The first inequality is Markov's inequality applied inside the tail integral.
A union bound makes these inequalities simultaneous over the prespecified admissible pairs.  Taking their minimum and then taking the maximum with one proves \eqref{eq:integrated-monte-carlo-ucb}.  On the event $M_\Sigma^-\le\widehat M_{\Sigma,U}$, Corollary~\ref{cor:gaussian-calibration} bounds the unknown-null factor by $M_{\Hnull}^-\le M_\Sigma^-\le\widehat M_{\Sigma,U}$.  Equation~\eqref{eq:integrated-finite-sample-level} then follows from Corollary~\ref{cor:estimated-integrated-factor}.
\end{proof}

Equivalently, the admissibility condition is $r_\star^2<\{1+4s(s-1)\}^{-1}$.  For $s=4$, this is $r_\star^2<1/49$ with tail remainder $C_{4,\Sigma}/(3h^3)$; for $s=2$, the weaker condition $r_\star^2<1/9$ comes with the slower $h^{-1}$ remainder.  Values of $s$ closer to one allow larger $r_\star^2$ but further slow the remainder.  This restriction concerns only the moment bound used here; Corollary~\ref{cor:gaussian-calibration} remains valid for every positive-definite $\Sigma$.

\paragraph{Computation for Figure~\ref{fig:adjustment-numerics}.}
Both models use $m=5000$, target $\alpha=0.1$, 4750 nulls, and 250 alternatives.  The Gaussian row uses $10{,}000$ testing samples per point; the stratified Monte Carlo sizes for the aligned-sign row are given below.

\emph{Global equicorrelation.}
The top row uses
\[
 W_i=\mu_i+\sqrt\rho\,G_0+\sqrt{1-\rho}\,Z_i,
 \qquad
 \rho\in\{0,0.005,0.010,0.020,0.030,0.050\},
\]
where $G_0,Z_1,\ldots,Z_m$ are independent standard normal variables, null means are zero, and alternative means are $4.5$.
For every positive $\rho$, both adjustments use the same $B=200{,}000$ independent centered Gaussian calibration draws; each displayed branch has failure budget $\eta=0.001$.  The calibration sample is held fixed across the testing repetitions.  At $\rho=0$, exact conditional sign flips justify input level $0.1$ for all three procedures.

For computation, let $T_X=\sum_{j=1}^mX_j$, $X_{(1)}=\min_jX_j$, and
\[
 c_{\rho,m}=\frac{2\rho}{(1-\rho)\{1+(m-1)\rho\}}.
\]
For negative coordinates the relevant expression is convex, so its maximum occurs at $X_{(1)}$.  Gaussian regression therefore gives
\[
 \begin{aligned}
 L_\Sigma(X)
 &=\max\left\{0,\max_j c_{\rho,m}|X_j|(T_X-X_j)\right\},\\
 \Omega_\Sigma^-(X)
 &=\exp\left([c_{\rho,m}\{X_{(1)}^2-T_X X_{(1)}\}]_+
             \times\ind\{X_{(1)}<0\}\right).
 \end{aligned}
\]
Both statistics require $O(m)$ operations per draw.  For the integrated bound, we prespecify
\[
 \begin{aligned}
 \mathcal S&=\{2,2.25,2.5,2.75,3,3.25,3.5,4\},\\
 \mathcal C&=\{25,50,75,100,150,200,300,500,750,1000,1500,
 2000,3000,5000,7500,10000\}.
 \end{aligned}
\]
Here $J=|\mathcal S||\mathcal C|=128$ in the union bound; only pairs satisfying the moment condition in Corollary~\ref{cor:finite-sample-gaussian} enter the minimum.  Exceptional-event calibration uses
\[
 \Delta=\{0.0025,0.005,0.0075,0.01,0.015,0.02,0.025,0.03,
          0.04,0.05,0.06,0.075,0.09\},
\]
with failure probability $\eta/13$ assigned to each order-statistic bound.
For each $\delta\in\Delta$, compute $q_\delta=(0.1-\delta-0.001)e^{-\widehat\kappa_\delta}$ using the order-statistic bound $\widehat\kappa_\delta$ from Corollary~\ref{cor:finite-sample-gaussian}.
Set $\widehat\delta\in\arg\max_{\delta\in\Delta}q_\delta$, $\widehat\kappa=\widehat\kappa_{\widehat\delta}$, $q_{\rm tail}=q_{\widehat\delta}$, and $q_{\rm int}=(0.1-0.001)/\widehat M_{\Sigma,U}$.
For the five positive correlations, in increasing order, the realized values, rounded as shown, are
\[
\begin{array}{c|ccccc}
 \rho&.005&.010&.020&.030&.050\\ \hline
 \widehat M_{\Sigma,U}&1.305&1.497&1.905&2.378&3.685\\
 (s,h)&(4,75)&(4,75)&(4,100)&(3.25,200)&(2.5,1000)\\
 q_{\rm int}&.07584&.06613&.05198&.04162&.02686\\
 (\widehat\delta,\widehat\kappa)
 &(.020,1.12997)&(.025,1.55475)&(.040,1.99635)&(.040,2.47982)&(.050,3.09528)\\
 q_{\rm tail}&.02552&.01563&.00801&.00494&.00222
\end{array}
\]
\emph{Rare aligned signs.}
For the bottom row, with $n_0=4750$, draw independent $U_j\sim|N(0,1)|$ and, independently, let the null-sign vector satisfy
\[
 S=\begin{cases}
  \text{independent fair signs},&\text{with probability }1-\theta,\\
  (+1,\ldots,+1),&\text{with probability }\theta/2,\\
  (-1,\ldots,-1),&\text{with probability }\theta/2,
 \end{cases}
 \qquad W_j=S_jU_j.
\]
The alternatives are independent $N(5.5,1)$ variables, and we use
\[
 \theta\in\{0,0.005,0.01,0.02,0.04,0.08,0.12,0.18\}.
\]
Global sign symmetry gives standard-normal null margins, while the independent-sign component gives every sign vector positive probability.

For $\theta>0$, let $c_\theta=(1-\theta)2^{-n_0}$.  Direct conditioning shows that odds above one occur only for the all-positive and one-negative sign vectors.  Consequently,
\[
 \begin{gathered}
 \delta_\theta=\frac{\theta}{2}+(n_0+1)c_\theta,
 \qquad \Gamma_{\delta_\theta}=1,
 \qquad M_{\Hnull}^-=1+\frac{n_0\theta}{2},\\
 q_{\rm tail}=0.1-\delta_\theta,
 \qquad q_{\rm int}=\frac{0.1}{1+n_0\theta/2}.
 \end{gathered}
\]
At $\theta=0$, the signs are independent fair coins, so all three procedures instead use the exact level $0.1$.  Because the factors above are known exactly, their maximum needs no calibration failure budget.

For each positive $\theta$, we simulate $10{,}000$ samples from each mixture component.  Their means use weights $1-\theta$, $\theta/2$, and $\theta/2$, and their Monte Carlo variances use the squared weights.  At $\theta=0.18$, $q_{\rm tail}\approx0.01$, whereas $M_{\Hnull}^-=428.5$ and $q_{\rm int}\approx0.000233$; this large difference explains the bottom-row power gap.
\end{newrevision}

\begin{revision}
\section{Calibration with negative-control datasets}
\label{app:negative-control-calibration}
Apply the score rule to $B$ mutually independent all-null datasets, independent of $W$, to obtain $W^{(1)},\ldots,W^{(B)}$; separate negative-control experiments are one example.  Coordinate $j$ is the same feature throughout, so these are new datasets, not recalculations or resamples.  Coordinates within a vector may depend.  Set $\mathcal D_{\rm cal}=(W^{(1)},\ldots,W^{(B)})$.  When the labels $C_j$ below are prespecified, take $\mathcal D=\mathcal D_{\rm cal}$.

Let $C_j=c_j(|W_j|,W_{-j})\in\{1,\ldots,K_j\}$ be a finite label for information that may affect the sign, and let $C_j^{(b)}$ be its control version.  Assume the same $p_{jk}$ satisfies
\begin{align*}
 \Prb(W_j>0\mid |W_j|,W_{-j})
 &=p_{j,C_j}, && j\in\Hnull,\\*
 \Prb(W_j^{(b)}>0\mid |W_j^{(b)}|,W_{-j}^{(b)})
 &=p_{j,C_j^{(b)}}, && j=1,\ldots,m,\ b=1,\ldots,B.
\end{align*}
This condition says that $C_j$ retains all conditional sign information within each cell and that the same $p_{jk}$ governs the testing and control data.  A majority-sign label is one example, but controls alone do not ensure the condition.

For each fixed $j$, let $N_{jk}$ count controls in cell $(j,k)$ and $X_{jk}$ count their positive signs:
\[
 \begin{aligned}
 N_{jk}&=\sum_{b=1}^B\ind\{C_j^{(b)}=k\},
 &X_{jk}&=\sum_{b=1}^B\ind\{C_j^{(b)}=k,\ W_j^{(b)}>0\},\\
 X_{jk}\mid C_j^{(1)},\ldots,C_j^{(B)}
 &\sim\operatorname{Bin}(N_{jk},p_{jk}).
 \end{aligned}
\]
This uses independence across datasets, not coordinates; other coordinates' labels could reveal the sign and are not conditioned upon.

Choose $\eta_{jk}\ge0$ before seeing the signs, with $\sum_{j,k}\eta_{jk}\le\eta$.  An exact upper confidence limit is
\[
 U_{jk}=
 \begin{cases}
  \operatorname{Beta}^{-1}
  (1-\eta_{jk};X_{jk}+1,N_{jk}-X_{jk}),
    &N_{jk}>0,\ X_{jk}<N_{jk},\ \eta_{jk}>0,\\
 1,&\text{otherwise}.
 \end{cases}
\]
Here $\operatorname{Beta}^{-1}(a;r,s)$ is the $a$-quantile of $\operatorname{Beta}(r,s)$.  Conditional on coordinate $j$'s labels, $\Prb(U_{jk}<p_{jk})\le\eta_{jk}$ exactly, without a normal approximation.  The second branch covers empty or all-positive cells and $\eta_{jk}=0$.

For $A=\{p_{jk}\le U_{jk}\text{ for all }j,k\}$, the union bound gives $\Prb(A)\ge1-\eta$ without independence across cells.  Since scores are nonzero, $p/(1-p)$ is the positive-to-negative sign odds.  Define
\[
 \widehat\Gamma
 =\max_{j,k}\left\{1,\frac{U_{jk}}{1-U_{jk}}\right\},
 \qquad
 \Omega_j
 =\frac{p_{j,C_j}}{1-p_{j,C_j}}
 \le\frac{U_{j,C_j}}{1-U_{j,C_j}}
 \le\widehat\Gamma,
\]
where $U/(1-U):=+\infty$ at $U=1$; thus any such cell gives $\widehat\Gamma=\infty$ and no rejections.  On $A$ and for finite $\widehat\Gamma$, $\Prb(\Lambda_{\max}>\log\widehat\Gamma\mid\mathcal D)=0$, so Corollary~\ref{cor:estimated-gamma} applies with $\delta=0$.  This construction implements only the exceptional-event branch; in the combined notation one sets $\widehat M=\infty$.

If labels are learned on training data independent of calibration and testing, take $\mathcal D=(\mathcal D_{\rm train},\mathcal D_{\rm cal})$ and condition on training.  An estimator of auxiliary model quantities shared by calibration and testing needs a separate conditional-independence justification; otherwise, use separate fits.
\end{revision}

\end{document}